\documentclass[11pt,twoside]{amsart}
\usepackage[a4paper]{geometry}
\usepackage{amstext}
\usepackage{amsthm}
\usepackage{amssymb}
\usepackage{stmaryrd}
\usepackage[unicode=true,pdfusetitle,
 bookmarks=true,bookmarksnumbered=false,bookmarksopen=false,
 breaklinks=false,pdfborder={0 0 1},backref=false,colorlinks=false]
 {hyperref}

\makeatletter

\newcommand{\lyxmathsym}[1]{\ifmmode\begingroup\def\b@ld{bold}
  \text{\ifx\math@version\b@ld\bfseries\fi#1}\endgroup\else#1\fi}

\numberwithin{equation}{section}
\numberwithin{figure}{section}
\theoremstyle{plain}
\newtheorem*{thm*}{\protect\theoremname}
\theoremstyle{plain}
\newtheorem{thm}{\protect\theoremname}
\theoremstyle{plain}
\newtheorem{lem}[thm]{\protect\lemmaname}
\theoremstyle{remark}
\newtheorem{rem}[thm]{\protect\remarkname}
\theoremstyle{plain}

\theoremstyle{definition}
\newtheorem{defn}[thm]{\protect\definitionname}
\theoremstyle{definition}

\theoremstyle{plain}
\newtheorem{pro}[thm]{\protect\propositionname}
\theoremstyle{plain}
\newtheorem{asm}[thm]{\protect\assumptionsname}
\usepackage{amssymb,amsthm,amsmath,amsfonts,amscd}
\usepackage{graphicx}
\usepackage{url}
\usepackage{color}
\usepackage[active]{srcltx}
\usepackage[matrix,arrow]{xy}
\usepackage{mathrsfs}
\usepackage{enumitem}
\usepackage{amsopn} % provides \DeclareMathOperator
\usepackage{bbm} % additional fonts
\usepackage[normalem]{ulem} %underline etc.
\usepackage{cite}
\usepackage{hyperref}
\allowdisplaybreaks[1]
\usepackage[english]{babel}
\usepackage{bbm}
\usepackage{mhchem}
\usepackage{mdef}
\date{\today}
\makeatother

\providecommand{\corollaryname}{Corollary}
\providecommand{\definitionname}{Definition}
\providecommand{\examplename}{Example}
\providecommand{\lemmaname}{Lemma}
\providecommand{\remarkname}{Remark}
\providecommand{\theoremname}{Theorem}
\providecommand{\propositionname}{Proposition}
\providecommand{\assumptionsname}{Assumptions}

\begin{document}

\title[Stochastic Fokker-Planck equations]{Stochastic nonlinear Fokker-Planck equations}
\begin{abstract}
The existence and uniqueness of measure-valued solutions to stochastic nonlinear, non-local Fokker-Planck equations is proven. This type of stochastic PDE is shown to arise in the mean field limit of weakly interacting diffusions with common noise. The uniqueness of solutions is obtained without any higher moment assumption on the solution by means of a duality argument to a backward stochastic PDE.
\end{abstract}

\thanks{Acknowledgements: The authors acknowledge financial support by the the Max Planck Society through the Max Planck Research Group "Stochastic partial differential equations" and by the DFG through the CRC 1283 "Taming uncertainty and profiting from randomness and low regularity in analysis, stochastics and their applications."}

\author{Michele Coghi}
\address[M. Coghi]{Weierstrass Institute for Applied Stochastic, Mohrenstrasse 39, 10117 Berlin and Faculty of Mathematics,
	University of Bielefeld,
	33615 Bielefeld,
	Germany}
\email{michele.coghi@gmail.com}

\author{Benjamin Gess}
\address[B. Gess]{Max Planck Institute for Mathematics in the Sciences, Inselstrasse 22, 04103 Leipzig and Faculty of Mathematics,
University of Bielefeld,
33615 Bielefeld,
Germany}
\email{benjamin.gess@gmail.com}

\maketitle

\section{Introduction}

We consider the following stochastic nonlinear, non-local  Fokker-Planck equation\footnote{with Einstein summation convention, and $\partial_i$ being the partial derivative with respect to the space variable $x_i$.} on $[0,T]\times \mathbb{R}^d$,
\begin{equation}
\label{target equation}
	\left\{
	\begin{array}{l}
		\partial_t \mu
		= \partial^2_{i,j}(a^{ij}(t,x,\mu) \mu)
		- \partial_{i} (b^i(t,x,\mu) \mu)  
		- \partial_{i}(\sigma^{ik}(t,x,\mu)\mu) dW^k_t,\\
		\mu|_{t=0} = \mu_0,
	\end{array}
	\right.
\end{equation}
where $(\Omega, \mathcal{F}, (\mathcal{F}_t)_{t\in[0,T]}, \mathbb{P})$ is a filtered probability space, $W_t$ is a $d_1$-dimensional $(\mathcal{F}_t)_{t\in[0,T]}$-Brownian motion for some $d_1 \in \mathbb{N}$, $a = (a^{ij}): [0,T] \times \mathbb{R}^d \times \mathcal{M}(\mathbb{R}^d)  \to \mathbb{R}^{d\times d}$ is a function from the space of measures on $\mathbb{R}^d$ to the space of symmetric and non-negative definite matrices, $\sigma = (\sigma^{ik}): [0,T] \times \mathbb{R}^d \times \mathcal{M}(\mathbb{R}^d) \to \mathbb{R}^{d\times d_1}$ takes values in the space of $d \times d_1$ matrices and $b = (b^i): [0,T] \times \mathbb{R}^d \times \mathcal{M}(\mathbb{R}^d) \to \mathbb{R}^d$. We emphasize that the  coefficients $a^{ij},b^i,\sigma^{ik}$ depend non-locally and possibly nonlinearly on the solution $\mu$.

Since stochastic PDE of the type \eqref{target equation} describe the evolution of conditional distributions of solutions to McKean-Vlasov SDE with common noise (see below) it is natural to consider solutions $(\mu_t)_{t\in [0,T]}$ to  \eqref{target equation} taking values in the space of finite non-negative measures on $\mathbb{R}^d$. The main result of this paper is to establish the well-posedness of measure-valued solutions to \eqref{target equation}. % and consequently also to the SSCL \eqref{SSCL}. We have the following 
\begin{thm*}[see Theorem \ref{thm:uniqueness spde} and \ref{thm:existence spde} below]
	Let $\mu_0 \in \mathcal{M}(\mathbb{R}^d)$ be a non-negative measure. If the coefficients $a,b,\sigma$  are regular enough, then there exists a unique solution $\mu \in L^1_{\omega}C_t\mathcal{M}$ to equation \eqref{target equation} in the sense of Definition \ref{definition of solution}, below.
\end{thm*}
%As a consequence we also obtain well-posedness for the stochastic scalar conservation law \eqref{SSCL}.

Previously, the uniqueness of solutions to \eqref{target equation} was known only in the class of solutions to \eqref{target equation} admitting a square-integrable density with respect to the Lebesgue measure (e.g.\ Kurtz, Xiong \cite{kurtz1999particle}). This is in contrast to the deterministic case, where the uniqueness of measure-valued solutions has recently been shown based on duality arguments by Manita, Romanov, Shaposhnikov in \cite{MR3399526, MR3113428}. Following this approach, the existence of regular enough solutions to the dual equation, a parabolic PDE backwards in time, implies the uniqueness of solutions to \eqref{target equation} with $\sigma \equiv 0$. This line of argument becomes more challenging in the case of stochastic PDE since the dual equation becomes a backward stochastic PDE (BSPDE) and, therefore, has not previously been put to use in the case of stochastic PDE, such as  \eqref{target equation}. This is the purpose of the present work. In particular, the method employed here can be seen as a proof of principle of using duality arguments to derive the uniqueness of solutions for stochastic PDE.

\subsection*{Motivation and model}
The stochastic PDE \eqref{target equation} is linked to stochastic scalar conservation laws (SSCL) of the form
\begin{equation}
\label{SSCL}
	\left\{
	\begin{array}{lcl}
	du + \operatorname{div}(\sigma(x,u)u \circ dW) = 0 & \mbox{in} & \mathbb{R}^d \times (0, T),\\
	u = u_0 & \mbox{on} & \mathbb{R}^d \times \{0\}.
	\end{array}
	\right.
\end{equation}
Indeed, rewriting equation \eqref{SSCL} in It\^{o} form (cf.\ Appendix \ref{appendix:ito strato} below), yields
\begin{equation}\label{eqn:FP_from_SCL}
\partial_t u
+ \partial_{i} (\sigma^{i,k}(x,u) u) dW^k_t 
+ \partial_{i} (b^i(x,u) u) 
=  \frac12 \partial_{i,j}^2(\sigma^{i,k}(x,u)\sigma^{k,j}(x,u) u),
\end{equation}
which is of the same type as equation \eqref{target equation}. In particular, we notice that both first order and second order correction terms appear, and that they are both nonlocal in the variable $u$. For the exact definition of $b$, we refer to \eqref{lions:drift definition} below.

Stochastic scalar conservation laws and thus stochastic PDE of the type \eqref{target equation} arise in several applications. Examples are provided by the theory of mean field systems (see Sznitman \cite{sznitman1991topics} for an overview) and mean field games with common noise introduced by Lasry and Lions \cite{MR2269875, MR2271747, MR2295621}, with an extensive treatment given by Carmona, Delarue in \cite{MR3752669, MR3753660}. 
Consider the empirical law $L^N_t := \frac{1}{N} \sum_{j=1}^{N}\delta_{X_t^j}$ of the solution $(X^1, \dots, X^N): [0,T] \times \Omega \to \mathbb{R}^{d N}$ of the weakly interacting particle system
\begin{equation}
	\label{introduction:particles}
	dX_t^i = \sigma(X_t^i, \frac{1}{N} \sum_{j=1}^{N}\delta_{X_t^j}) \circ dW, 
	\quad X^i_t\vert_{t=0} =X^i_0,
	\quad \mbox{for } i = 1, \cdots, N,
\end{equation}
with initial conditions $(X^i_0)_{i\geq 1}$ independent and identically distributed.  The random measure $L^N$ converges, as $N \to \infty$, to a random measure $u$ which evolves according to \eqref{SSCL} (cf.\ Section \ref{section:interacting particles} below). %The convergence is usually understood in the space $\mathcal{P}(\mathbb{R}^d)$ of probability measures on $\mathbb{R}^d$ endowed with the Wasserstein metric. The function $\sigma$ is usually assumed to be Lipschitz continuous in both its variables.

The above mentioned convergence of (random) empirical measures is closely linked to the phenomenon of propagation of chaos and to McKean-Vlasov SDE (cf.\ e.g.\ \cite{MR0221595, MR0233437, MR896631, MR3513594}). More precisely, in the limit $N\to\infty$, solutions to \eqref{introduction:particles}  converge to the solution to the McKean-Vlasov SDE
\begin{equation}
	\label{introduction:mckean vlasov}
	dX_t = \sigma(X_t, \mathcal{L}(X_t\mid \mathcal{W})) \circ dW, 
	\quad X_t\vert_{t=0} =X^1_0,
\end{equation}
where $\mathcal{L}(X\mid \mathcal{W})$ is the conditional law of $X$ with respect to  $W$, as explained in detail in \eqref{eq:defn conditional law} below. Given a solution $X$ to \eqref{introduction:mckean vlasov} its conditional law $\mathcal{L}(X\mid \mathcal{W})$ then satisfies \eqref{SSCL}.

Noticeably, all of the particles in \eqref{introduction:particles} are subject to the same \textit{common noise}. For this reason, no averaging effect with respect to this noise is observed and it thus survives in the limit $N\to\infty$, leading to a stochastic PDE. %Hence, the measure which appears in \eqref{introduction:mckean vlasov} is a random measure. This measure is a solution to the SSCL \eqref{SSCL}.

\subsection*{Literature}
Stochastic scalar conservation laws have been the object of several studies. In the case that $\sigma(x,u)=\sigma(u(x))$, that is, coefficients $\sigma$ depending on $u$ in a local and spatially homogeneous way, this class of stochastic PDE was introduced by Lions, Perthame, Souganidis in \cite{MR3327520}. For linear, spatially inhomogeneous coefficients, the well-posedness of entropy solutions was shown by Friz, Gess in \cite{MR3519527}. The case of local, nonlinear coefficients was later generalized to spatially inhomogeneous coefficients by Lions, Perthame, Souganidis in \cite{MR3274890}, Gess, Souganidis in \cite{MR3351442} and to include second order operators by Gess, Souganidis in \cite{MR3682120} and Fehrman, Gess in \cite {FG17}. Qualitative properties of solutions, such as regularity and finite speed of propagation has been considered by Gassiat, Gess in \cite{GG16} and Gassiat, Gess, Lions, Souganidis in \cite{GGLS18}. 

In a recent article \cite{barbu2018nonlinear}, Barbu and R\"{o}ckner treat McKean-Vlasov SDE when the dependence on the law is local, proving, roughly speaking, that if there is a solution to the scalar conservation law \eqref{SSCL}, then there also is a solution to the McKean-Vlasov equation \eqref{introduction:mckean vlasov}.

%In the case of thermal noise, that is, considering \eqref{introduction:particles} with independent noise $W^i$ on each particle, instead of a common noise, from mean-field theory (e.g.\ Sznitman \cite{sznitman1991topics}) it is well known that the empirical measure converges to a deterministic probability measure, which is a weak solution to a non-linear Fokker-Planck equation of the form of equation \eqref{target equation}, with $\sigma \equiv 0$. General conditions for the existence and uniqueness of solutions in this case have been recently studied in \cite{MR3399526, MR3113428, MR2483999, MR2640966, MR3391701}. The proof of uniqueness in the deterministic case is based on a duality method, called Holmgren principle, which we will later also apply to the stochastic equation.

The existence and uniqueness of solutions to deterministic non-linear Fokker-Planck equations of the form \eqref{target equation} with $\sigma \equiv 0$ has been recently studied by several authors in \cite{MR3399526, MR3113428, MR2483999, MR2640966, MR3391701}. 

%For the linear equation ($a$ and $b$ independent of $\mu$), the idea is to take a strong $C^2$-solution $f_t$ of the backward equation
%\begin{equation*}
%	\partial_t f + a^{i,j}\partial_{x_i}\partial_{x_j} f + b^i\partial_{x_i} f = 0, \quad f|_{s=t} = \psi,
%\end{equation*}
%where $\psi \in C_0^\infty(\mathbb{R}^d)$. Then one multiplies both sides by a solution $\mu$ and obtains
%\begin{equation}
%	\label{introduction:duality}
%	\langle \mu_t, \psi \rangle = \langle \mu_t, f_t \rangle = \langle \mu_0, f_0 \rangle.
%\end{equation}
%If the backward equation admits a solution, then the value of $\mu_t$ tested against any $\psi$ is uniquely determined, which means that the solution $\mu$ must be unique.

As mentioned above, to the best of our knowledge, the uniqueness of solutions to non-local stochastic PDE of the type \eqref{target equation} is known only in the class of solutions $\mu$ such that for each $t>0$, $u(t)$ is absolutely continuous with respect to Lebesgue measure and has a density in $L^2(\R^d)$ (cf.\ Kurtz, Xiong \cite[p.\ 115]{kurtz1999particle}). Under more restrictive conditions, either on the class of solutions or on the coefficients of \eqref{target equation}, the well-posedness of solutions to SPDE of the type \eqref{target equation} had been previously considered by Dawson, Vaillancourt in \cite{DV95}, where the uniqueness of solutions has been obtained by several methods, e.g. by constructing a dual process, by coupling arguments and by the Krylov-Rozovskii "variational" approach to SPDE.

Motivated from fluid dynamics in vorticity form, also signed measure-valued solutions to SPDE of the type \eqref{target equation} have been considered in the literature. We refer to R\'emillard, Vaillancourt \cite{RV14}, Kotelenez \cite{K10}, Kotelenez, Seadler \cite{KS12}, Amirdjanova, Xiong \cite{AX06} and the references therein. Again, uniqueness of solutions was obtained only under more restrictive assumptions. 

%. In \cite{kurtz1999particle} an infinite system of stochastic interacting particles is considered. The particles interact through their empirical measure, which is shown to be, in the limit for the number of particles that grows to infinity, the unique solution of an equation similar to \eqref{target equation}. To prove the well-posedness of equation \eqref{target equation} in the space of measures, we use a technique that is similar to the one used in \cite{kurtz1999particle} to prove existence and uniqueness of \eqref{target equation} in $L^2$.

\subsection*{Outline of the proof}

The proof of uniqueness of solution to \eqref{target equation} put forward in the present work relies on the well-posedness for the Lagrangian characteristics
\begin{equation}\label{eqn:MCKV}
\left\{
\begin{array}{ll}
dX_t &= b(X_t, \mu_t)dt + \sigma(X_t, \mu_t) dW_t + \alpha(X_t, \mu_t)dB_t, 
\quad \mu_t := \mathcal{L}(X_t\mid \mathcal{W})\\
X_t|_{t=0} &=X_0,
\end{array}
\right.
\end{equation}
where $W$ and $B$ are independent Brownian motions, $X_0$ is an independent random variable on some filtered probability space $(\Omega, \mathcal{F}, (\mathcal{F}_t)_{t\geq 0}, \mathbb{P})$ and $\alpha:=[2a-\s^T\s]^\frac{1}{2}$.

The proof of uniqueness of solutions to \eqref{target equation} proceeds by freezing the coefficients of equation \eqref{target equation} and proving the uniqueness of solutions to the resulting linear equation
\begin{equation}\begin{split}
\label{target equation_frozen}
		\partial_t \bar\mu
		&= \partial^2_{i,j}(\bar a^{ij}(t,x) \bar\mu)
		- \partial_{i} (\bar b^i(t,x) \bar\mu)  
		- \partial_{i}(\bar \sigma^{ik}(t,x)\bar\mu) dW^k_t.
\end{split}\end{equation}
At this point, in contrary to the previous work \cite{kurtz1999particle}, the uniqueness of measure-valued solutions to \eqref{target equation_frozen} has to be shown, while \cite{kurtz1999particle} was restricted to solutions allowing square-integrable densities. Here the above mentioned duality argument comes into play, on which we comment in more detail below. The uniqueness of solutions to \eqref{target equation_frozen} then implies that each solution $\mu$ to \eqref{target equation} is given as the conditional law $\mathcal{L}(X\mid \mathcal{W})$ of a solution to \eqref{eqn:MCKV}. Therefore, uniqueness to \eqref{eqn:MCKV} implies the uniqueness for \eqref{target equation}.

In order to prove the uniqueness of solutions to \eqref{target equation_frozen}, we employ a duality argument, which leads to the
%Thanks to the theory of backward stochastic differential equations (BSDE) we can apply the Holmgren duality principle to the stochastic equation \eqref{target equation}.
%We will use the following 
backward stochastic PDE
\begin{equation}
\label{introduction:bspde}
\partial_t f 
= - a^{i,j} \partial_{i,j}^2 f
- b^i \partial_{i} f 
- \sigma^{i,k} \partial_{i} v^k
+ v^k dW^k, \quad
f_T = \varphi,
\end{equation}
where the terminal condition is a sufficiently smooth random test function. We emphasize that in the case of stochastic scalar conservation laws \eqref{SSCL}, and equivalently \eqref{eqn:FP_from_SCL} in It\^{o} form, we have 
\begin{equation*}
	a^{i,j}(u) = \frac12 \sigma^{i,k}(u)\sigma^{k,j}(u),
\end{equation*} 
which implies that \eqref{introduction:bspde} is degenerate. For background on degenerate backward stochastic PDE we refer to \cite{MR1151549, MR3017034, MR2158497, MR1918539}. In order to invoke the duality argument for measure-valued solutions, we require classical solutions to \eqref{introduction:bspde} which can be obtained based on \cite{MR3017034} by Du, Tang, Zhang, and Sobolev embedding. It then follows (cf.\ Lemma \ref{duality lemma} below) that
\begin{equation}
	\E\langle \mu_t, \varphi \rangle = \E\langle \mu_t, f_t \rangle = \E\langle \mu_0, f_0 \rangle,
\end{equation}
which implies the uniqueness of measure-valued solutions to \eqref{target equation_frozen}. We also refer to Zhou \cite{MR1151549} and Diehl, Friz, Stannat \cite{DFS17} for results on the duality of stochastic PDE and backward stochastic PDE.

\subsection*{Structure of the paper}
In Section \ref{section: preliminaries} we will set the notation. In Section \ref{section: McKean} we analyze the Lagrangian dynamics. Section \ref{section:lienear spde} is devoted to the proof of well-posedness of linear SPDEs using Holmgren principle. In Section \ref{section: main results} we prove well-posedness for the non-local SPDE \eqref{target equation}.

\section{Notations and assumptions}
\label{section: preliminaries}

We fix two numbers $d, d_1 \in \mathbb{N}$. We will use the following notational conventions for the indices: $i,j \in \{1, \dots, d\}$ and $k\in \{1, \dots, d_1\}$.

For any multi-index $\alpha = (\alpha_1, \dots, \alpha_d)$, we set
\begin{equation*}
	D^{\alpha} = \left( \frac{\partial}{\partial x_1}\right)^{\alpha_1} \left( \frac{\partial}{\partial x_2}\right)^{\alpha_2} \cdots \left( \frac{\partial}{\partial x_d}\right)^{\alpha_d}
\end{equation*}
and $\vert \alpha \vert = \alpha_1 + \cdots + \alpha_d$.

Let $C_c^\infty$ and $C^n$ be the set of infinitely differentiable differentiable real-valued functions of compact support defined on $\mathbb{R}^d$ and the set of $n$ times continuously differentiable functions on $\mathbb{R}^d$ such that
\begin{equation*}
	\Vert \varphi \Vert_{C^n} := \sum_{\vert \alpha \vert \leq n} \sup_{x\in\mathbb{R}^d} \vert D^\alpha \varphi \vert < +\infty.
\end{equation*}

Let $\operatorname{Lip}_1$ be the space of Lipschitz continuous functions in $C^0$, such that
\begin{equation*}
	\Vert \varphi \Vert_{C^0}, \sup_{x \neq y \in \mathbb{R}^d} \frac{\vert \varphi(x) - \varphi(y) \vert}{\vert x - y \vert} \leq 1.
\end{equation*}

For a function $f \in C([0,T]; \mathbb{R}^d)$ we call $\Vert f \Vert_{\infty}$ its supremum norm.

For $p>1$ and an integer $m\geq 0$, we let $W^{m, p} = W^{m, p}(\mathbb{R}^d; \mathbb{R})$ be the Sobolev space of real-valued functions on $\mathbb{R}^d$ with finite norm
\begin{equation*}
	\Vert f\Vert_{W^{m, p} }:= \left( \sum_{\vert \alpha \vert \leq n} \int_{\mathbb{R}^d} \vert D^{\alpha}f(x) \vert^p dx \right)^\frac{1}{p} < \infty.
\end{equation*}
In the same way, denote by $W^{m,p}(\mathbb{R}^{d_1}) = W^{m, p}(\mathbb{R}^d;\mathbb{R}^{d_1})$ the Sobolev space of $d_1$-dimensional vector-valued functions on $\mathbb{R}^d$, equipped with the norm 
\begin{equation*}
	\Vert v \Vert_{W^{m,p}} := \left( \sum_{k=1}^{d_1}\Vert v^k \Vert_{W^{m,p}}^p  \right)^{\frac1p} <\infty.
\end{equation*}

We call $\mathcal{M}^{\pm}(\mathbb{R}^d)$ the space of finite signed measures on $\mathbb{R}^d$. On this space we define the total variation norm
\begin{equation*}
\Vert \mu \Vert_{TV} = \sup_{\Vert \varphi \Vert_{C^0} \leq 1} \langle \mu, \varphi \rangle.
\end{equation*}
For $r>0$, we define $\mathcal{M}^{\pm}_{\leq r} := \{ \mu \in \mathcal{M}^{\pm}(\mathbb{R}^d) \mid \Vert \mu \Vert_{TV} \leq r \}$.

We call $\mathcal{M}(\mathbb{R}^d) \subset \mathcal{M}^{\pm}(\mathbb{R}^d)$ (resp. $\mathcal{P}(\mathbb{R}^d)$) the space of finite positive (resp. probability) measures on $\mathbb{R}^d$.

For $r>0$, we call $\mathcal{M}_r(\mathbb{R}^d)$ the space of measures in $\mathcal{M}(\mathbb{R}^d)$ with total variation equal to $r$, namely
\begin{equation*}
\mathcal{M}_r(\mathbb{R}^d) =
\left\{
\mu \in \mathcal{M}(\mathbb{R}^d) \mid \Vert \mu \Vert_{TV} = r 
\right\}.
\end{equation*}
It is worth mentioning that $\mathcal{M}_1(\mathbb{R}^d) = \mathcal{P}(\mathcal{R}^d)$. We endow $\mathcal{M}_r(\mathbb{R}^d)$ with the Kantorovic-Rubinstein norm
\begin{equation*}
\|\mu\| := \sup_{\varphi \in \operatorname{Lip}_1} \left(\int_{\mathbb{R}^d} \varphi d\mu\right),
\quad \forall \nu \in \mathcal{M}_r(\mathbb{R}^d)
\end{equation*}
and let $\rho$ be the induced metric. On $\mathcal{M}_r$ we consider the Borel $\sigma$-algebra induced by $\rho$.

From \cite[Theorem 8.3.2]{MR2267655}, we have that the metric $\rho$ metrizes the weak convergence of measures. Moreover, the space $(\mathcal{M}(\mathbb{R}^d), \rho)$ is complete and separable, see \cite[Theorem 8.9.4]{MR2267655}.

On a filtered probability space $(\Omega, \mathcal{F}, (\mathcal{F}_t)_{t\geq 0}, \mathbb{P})$, let $W$ be a $d_1$-dimensional Brownian motion. Let $(\mathcal{W}^s_t)_{t\geq s}$ be the completion of the filtration generated by the increments of $W$ starting from $s$, namely the completion of
\begin{equation*}
\mathcal{W}^s_t := \sigma( W_r - W_s \; : \; s\leq r\leq t), \quad \forall s\leq t \in [0,\infty).
\end{equation*}
To simplify the notation, we omit the dependence from the starting time when it is zero, i.e. $\mathcal{W}_t := \mathcal{W}_t^0$. Moreover, we set $\mathcal{W} : = \vee_{t\geq0} \mathcal{W}_t$. It follows from the independence of the increments of the Brownian motion, that $\mathcal{W} = \mathcal{W}_t \vee( \vee_{T\geq t}\mathcal{W}_T^t)$, for all $t \in [0, T]$. Here with $\mathcal{F}\vee \mathcal{G}$ we indicate the $\sigma$-algebra generated by the union of the two $\sigma$-algebras.
\begin{defn}
	\label{defn:filtration}
	We say that the filtration $(\mathcal{F}_t)_{t\geq 0}$ is compatible with the Brownian motion $W$, if $W$ is $(\mathcal{F}_t)_{t\geq 0}$-adapted and if there exists a complete filtration $(\mathcal{G}_t)_{t\geq 0}$, such that, for every $t\geq 0$, $\mathcal{G}_t$ is independent from $\mathcal{W}_t$ and
	\begin{equation}
	\mathcal{F}_t = \mathcal{W}_t \vee \mathcal{G}_t.
	\end{equation} 
\end{defn} 

\begin{rem}
Given an $(\mathcal{F}_t)_{t\geq 0}$-adapted stochastic process $X:[0,T] \times \Omega \to \mathbb{R}^d$, we have
\begin{equation}\label{conditional equivalence}
\mathbb{E}\left[ X_t \mid \mathcal{W} \right] = \mathbb{E}\left[ X_t \mid \mathcal{W}_t \right], \quad \mathbb{P}-a.s.
\end{equation}
This is a consequence of Lemma \ref{appendix lemma:conditioning}.
\end{rem}

Moreover, we call $\mathcal{L}(X \mid \mathcal{W})$ the family $(\mathcal{L}(X_t \mid \mathcal{W}))_{t\geq 0} \subset \mathcal{P}(\mathbb{R}^d)$, such that for every $t\in[0,T]$ and $\varphi \in C^0$, we have
\begin{equation}
\label{eq:defn conditional law}
\langle \mathcal{L}(X_t \mid \mathcal{W}), \varphi \rangle := \mathbb{E}\left[\varphi(X_t) \mid \mathcal{W} \right] = \mathbb{E}\left[\varphi(X_t) \mid \mathcal{W}_t \right], \quad \mathbb{P}-a.s.
\end{equation}

\begin{rem}
    In the following we always assume that the filtration $(\mathcal{F}_t)_{t\geq 0}$ is compatible with the Brownian motion $W$ in the sense of Definition \ref{defn:filtration}. The leading example is $\mathcal{F}_t$ being the sigma algebra $\mathcal{W}_t \vee \mathcal{B}_t \vee \mathcal{H}$, where $(\mathcal{B}_t)_{t\geq 0}$ is the completion of the filtration of a Brownian motion $B$, independent from $W$, and $\mathcal{H}$ is the $\sigma$-algebra generated by the initial condition. Equality \eqref{conditional equivalence} is still true in this case.
	
	Moreover, we will always assume that the space $(\Omega, \mathcal{F}_0, \mathbb{R})$ is atomless. This implies that, given a metric space $E$ and a probability $\mu \in \mathcal{P}(E)$, we can always construct a random variable $X:\Omega \to E$, with $X=\mu$. See \cite[Proposition 9.1.11]{bogachev2007measure}
\end{rem}

Denote by $L^1_{\omega, t}\mathcal{M}_r$ the space of $(\mathcal{W}_t)_{t\geq 0}$-adapted, $(\mathcal{B}(\mathbb{R}_{+}) \times \mathcal{F})$-measurable processes $\mu : [0, T] \times \Omega \to  \mathcal{M}_r(\mathbb{R}^d)$ such that
\begin{equation*}
	\mathbb{E}\left[\int_{0}^{T}\|\mu_t\| dt \right] < +\infty.
\end{equation*}
\begin{rem}
	The space $L^1_{\omega, t}\mathcal{M}_r$ is complete. Indeed: Given a Cauchy sequence $(\mu^n)_{n\in \mathbb{N}} \subset L^1_{\omega, t}\mathcal{M}_r$ there is a subsequence $(\mu^{n_k})_{k\in \mathbb{N}}$ which is almost surely a Cauchy sequence in $(\mathcal{M}_r, \|\cdot\|)$. Since $\mathcal{M}_r$ is complete, there exists a null set $N\subset \Omega$, such that, for all $\omega \in N^c$, there exists $\mu(\omega) \in \mathcal{M}_r$ such that $\|\mu^{n_k}(\omega)-\mu(\omega)\| \rightarrow 0$ as $k\to \infty$. Adaptedness and joint measurability of $\mu$ follows from the respective properties of  $\mu^{n_k}$. Since the norm $\|\cdot\|$ is bounded, dominated convergence concludes the argument. 
\end{rem}
Denote by $L^1_{\omega}C_t\mathcal{M}_r$ the space of $(\mathcal{W}_t)_{t\geq 0}$-adapted continuous processes $\mu : [0, T] \times \Omega \to  \mathcal{M}_r(\mathbb{R}^d)$ such that there exists $\nu \in \mathcal{M}_r$ that satisfies
\begin{equation*}
\mathbb{E}\left[\sup_{t\in[0,T]}\|\mu_t\|  \right] < +\infty.
\end{equation*}

Notice that $L^1_{\omega}C_t\mathcal{M}_r \subset L^1_{\omega, t}\mathcal{M}_r$.

For $p\geq 1$, denote by $L^p_{t,\mathcal{F}}$ the space of $p$-integrable, $(\mathcal{F}_t)_{t\geq0}$-adapted stochastic processes on $\mathbb{R}^d$. We denote by $C_tL_{\omega}^1 := C([0,T]; L^1(\Omega, \mathbb{R}^d))$.

\section{McKean-Vlasov stochastic differential equation}
\label{section: McKean}
In this section we discuss the well-posedness of a McKean-Vlasov SDE. Let $a$, $b$ and $\sigma$ be measurable functions as in the introduction. Throughout this section the following assumptions are in force.

\begin{asm}
	\label{assumptions:particles}
	There is an $r>0$ such that
	\begin{enumerate}[label=(\roman*), ref=\ref{assumptions:particles} (\roman*)]
		%[label=\ref{assumptions:particles}.\roman*)]%[label=\textbf{A.\arabic*}]

		\item \label{assumptions:Lipschitz} (Uniform Lipschitz continuity) There exists a constant $K > 0$ such that
		\begin{align*}
		\Vert & a(t,x,\mu) - a(t,x^\prime,\mu^\prime) \Vert 
		+ \Vert \sigma(t,x,\mu) - \sigma(t,x^\prime,\mu^\prime) \Vert
		+ \vert b(t,x,\mu) - b(t,x^\prime,\mu^\prime) \vert \\
		& \leq K\left(  \vert x - x^\prime \vert + \rho(\mu, \mu^\prime)\right),
		\end{align*}
		
		for all $\mu, \mu^\prime \in \mathcal{M}_r(\mathbb{R}^d)$, $t\in [0,T]$ and $x, x^\prime \in \mathbb{R}^d$.
		
		\item \label{assumptions:bounded} (Uniform boundedness) There exists a constant $K > 0$ such that
		\begin{align*}
		\Vert & a(t,x,\mu)\Vert 
		+ \Vert \sigma(t,x,\mu) \Vert
		+ \vert b(t,x,\mu) \vert 
		\leq K,
		\end{align*}
		for all $\mu \in \mathcal{M}_r(\mathbb{R}^d)$, $t\in [0,T]$ and $x \in \mathbb{R}^d$.
		
		\item \label{assumptions:parabolicity} (Parabolicity) For each $(t,x,\mu) \in [0, T] \times \mathbb{R}^d \times \mathcal{M}_r(\mathbb{R}^d) $,
		\begin{equation*}
		[2a^{ij}(t,x,\mu) - \sigma^{ik}\sigma^{jk}(t,x,\mu)] \xi^i\xi^j \geq 0, \quad \forall \xi \in \mathbb{R}^d.
		\end{equation*}
	\end{enumerate}

\end{asm}

From now on assume that Assumption \ref{assumptions:particles} is satisfied. Let $(\Omega, \mathcal{F}, (\mathcal{F}_t)_{t\geq 0}, \mathbb{P})$ be a filtered probability space and $W$ a $d_1$-dimensional Brownian motion on this space, which is compatible with $(\mathcal{F}_t)_{t\geq0}$ in the sense of Definition \ref{defn:filtration}. Let $B$ be a $d$-dimensional $(\mathcal{F}_t)_{t\geq0}$-adapted Brownian motion independent of $W$. Moreover, assume that $X_0: \Omega \to \mathbb{R}^d$ is an $\mathcal{F}_0$-measurable random variable. Notice that $X_0$ is independent of $W$ and $B$.

We set
\begin{equation}
\label{definition alpha}
\alpha(t,x,\mu) := \left[2a(t,x,\mu) - \sigma^T(t,x,\mu)\sigma(t,x,\mu)\right]^{\frac{1}{2}}
	\quad \forall t \in[0,T], \; x\in \mathbb{R}^d ,  \; \mu\in \mathcal{M}(\mathbb{R}^d).
\end{equation}
It follows from Assumption \ref{assumptions:parabolicity} that $\alpha(t, x, \mu)$ is well defined as a symmetric matrix. Moreover, $\alpha$ is Lipschitz continuous and bounded in its variables $\mu$ and $x$, namely, there exists a constant $K > 0$, possibly different than before, such that for all $t \in[0,T]$, $x,x^\prime \in \mathbb{R}^d$ and $\mu, \mu^\prime\in \mathcal{M}_r(\mathbb{R}^d)$, 
\begin{equation*}
\Vert  \alpha(t,x,\mu) - \alpha(t,x^\prime,\mu^\prime) \Vert 
\leq K\left(  \vert x - x^\prime \vert + \rho(\mu, \mu^\prime)\right).
\end{equation*}
We consider the McKean-Vlasov SDE
\begin{equation}\label{sde system}
\left\{
\begin{array}{l}
dX_t = b(t,X_t,\mu_t)dt + \sigma(t,X_t,\mu_t) dW_t + \alpha(t,X_t,\mu_t)dB_t,\\
X_t|_{t=0} = X_0,\\
\mu_t  := r\mathcal{L}(X_t \mid \mathcal{W}).
\end{array}
\right.
\end{equation}

%Note that, when $r = 1$, then $\mu = \mathcal{L}(X_t \mid \mathcal{W})$.

\begin{defn}
	\label{definition solution sde}
	Let $X_0 : \Omega \rightarrow \mathbb{R}^d$, be $\mathcal{F}_0$-measurable, $r>0$ and define $\mu_0 := r\mathcal{L}(X_0)$. We say that a stochastic process $(X, \mu) : [0,T] \times \Omega \rightarrow \mathbb{R}^d \times \mathcal{M}_r$ is a solution to the McKean-Vlasov equation \eqref{sde system} with initial condition $X_0$, if
	\begin{enumerate}
		\item $X$ is $(\mathcal{F}_t)_{t\in[0,T]}$-adapted and time-continuous.
		\item $\mu \in L^1_{\omega, t}\mathcal{M}_r$ and for all $t\in [0,T]$,
		\begin{equation*}
		\mu_t  = r\mathcal{L}(X_t \mid \mathcal{W}),
		\quad \mathbb{P}-a.s.
		\end{equation*}
		\item The following integral equation is satisfied, namely, $\forall t \in [0, T]$,
		\begin{equation*}
		X_t = X_0 + \int_{0}^{t}b(s,X_s,\mu_s)ds + \int_{0}^{t}\sigma(s,X_s,\mu_s) dW_s + \int_{0}^{t}\alpha(s,X_s,\mu_s)dB_s, \quad \mathbb{P}-a.s.
		\end{equation*}
%		\item we have
%		\begin{equation*}
%		\forall \varphi \in C^0(\mathbb{R}^d), 
%		\quad \forall t \in [0,T], 
%		\quad \langle \mu_t, \varphi \rangle = r\mathbb{E}[\varphi(X_t)\mid \mathcal{W}],
%		\quad \mathbb{P}-a.s.
%		\end{equation*}
	\end{enumerate}

\end{defn}
We obtain the following well-posedness result for equation \eqref{sde system}. 

\begin{theorem}\label{wellposedness sde}
	Fix $r>0$ and assume Assumption \ref{assumptions:particles}. Let $X_0 : \Omega \rightarrow \mathbb{R}^d$, be $\mathcal{F}_0$-measurable.
	Then, there exists a unique solution $(X, \mu)$ to equation \eqref{sde system} in the sense of Definition \ref{definition solution sde}. In addition, the solution satisfies $\mu \in L^1_{\omega}C_t\mathcal{M}_r$.
	
%	Moreover, if $(X^1, \mu^1), (X^2, \mu^2)$ are two solutions to \eqref{sde system} associated with initial conditions $X_0^1$ and $X_0^2$ respectively, then there exists a constant $C := C(T) > 0$, such that
%		\begin{equation*}
%		\mathbb{E}\sup_{t\in[0,T]}\rho(\mu^1_t, \mu^2_t) \leq C \rho(\mu^1_0, \mu^2_0),
%		\end{equation*}
%		with $\mu^i_0  = r\mathcal{L}(X^i_0)$, for $i=1,2$.
\end{theorem}
\begin{proof}
	This well-posedness result is a direct consequence of \cite[Theorem 2.3]{kurtz1999particle}. However, we provide an alternate proof here avoiding the infinite interacting particle system used in \cite{kurtz1999particle}, but rather studying equation \eqref{sde system} directly.
	To prove the existence of a solution we start with a stochastic process $\mu \in L^1_{\omega,t}\mathcal{M}_r$ and we freeze the coefficients in \eqref{sde system}, to obtain the following equation
\begin{equation}
\label{linear sde}
\left\{
\begin{array}{l}
dX^{\mu}_t = b(t, X_t^\mu, \mu_t)dt + \sigma(t, X_t^\mu, \mu_t) dW_t + \alpha(t, X_t^\mu, \mu_t)dB_t\\
X^\mu_0 = X_0. \\
\end{array}
\right.
\end{equation}

	The coefficients $b(t, x, \mu_t), \sigma(t, x, \mu_t), \alpha(t, x, \mu_t)$ are progressively measurable, Lipschitz continuous and bounded. Hence, there exists a unique time-continuous $(\mathcal{F}_t)_{t\geq 0}$-adapted solution $(X^\mu_t)_{t\in [0,T]}$ to equation \eqref{linear sde}, see \cite[Theorem 3.1.1]{MR3410409}.
	%By using the Burkholder-Davis-Gundy inequality and the boundedness of the coefficents, we can conclude that $X^\mu \in L^1_{\omega}C_t$.

We define the following operator
\begin{equation}
\label{defn:contraction operator}
\begin{array}{cccc}
\Phi : & L^1_{\omega,t}\mathcal{M}_r & \to & %L^1_{\omega}C_t\mathcal{M}_r \subset
L^1_{\omega,t}\mathcal{M}_r\\
& \mu & \mapsto & r \mathcal{L}(X^\mu \mid \mathcal{W})
\end{array}
\end{equation}
and we will prove that its iterates $\Phi^k$ for $k$ large enough are contractions with respect to the metric
\begin{equation*}
	d(\mu, \nu) : =
	\mathbb{E}\left[\int_0^T\rho(\mu_t, \nu_t)dt\right], \quad \forall \mu, \nu \in L^1_{\omega,t}\mathcal{M}_r.
\end{equation*}
%First we must show that $\Phi$ is indeed well-defined. Let $r\delta_0 \in \mathcal{M}_r$ be the Dirac measure in zero multiplied by $r$. We compute the following, $\forall t\in [0,T], \mu \in L^1_{\omega,t}\mathcal{M}_r$,
%\begin{equation*}
%\rho(\Phi(\mu)_t, r\delta_0) 
%= \sup_{\varphi \in \operatorname{Lip}_1}\vert \mathbb{E}\left[ r\varphi(X_t^\mu) - r\varphi(0)\mid \mathcal{W} \right] \vert
%\leq r\mathbb{E}\left[\vert X^\mu_t \vert \mid \mathcal{W} \right],
%\quad \mathbb{P}-a.s.
%\end{equation*}
%which implies
%\begin{equation*}
%	\mathbb{E}\essup_{t\in[0,T]} \rho(\Phi(\mu)_t, r\delta_0) 
%	\leq \mathbb{E}\essup_{t\in[0,T]}r\mathbb{E}\left[\vert X^\mu_t \vert \mid \mathcal{W} \right]
%	\leq r\mathbb{E}\Vert X^\mu \Vert_{\infty}
%	< +\infty,
%\end{equation*}
%where the last inequality follows because $X_t^\mu \in L^1_{\omega}C_t$. 
Let $s,t \in [0,T]$, $p \in [1, \infty)$, and $\mu \in L^1_{\omega,t}\mathcal{M}_r$, we have
%using dominated convergence for conditional expectations, we have for $t_n \to t$ as $n\to\infty$,
\begin{equation*}
\mathbb{E}\rho(\Phi(\mu)_t, \Phi(\mu)_{s})^p \leq r \mathbb{E}\vert X^\mu_t - X^\mu_{s} \vert^p.
\end{equation*}
Standard estimates on the solutions of SDEs and Kolmogorov's continuity theorem imply that the process $\Phi(\mu)$ has a modification which is time continuous with respect to the weak topology, which is induced by $\rho$. Hence, $\Phi(\mu) \in L^1_{\omega}C_t\mathcal{M}_r \subset L^1_{\omega,t}\mathcal{M}_r$.

We proceed by proving that $\Phi$ is a contraction on $(L^1_{\omega,t}\mathcal{M}_r, d) $. By the definition of the Kantorovich-Rubinstein metric and using the conditional Jensen inequality, we have for each $t \in [0,T]$,
\begin{align*}
	\mathbb{E}\int_{0}^t\rho(\Phi(\mu)_s, \Phi(\nu)_s)ds
	= &\mathbb{E}\left[\int_{0}^t\sup_{\varphi \in \operatorname{Lip}_1}\mathbb{E}\left[r\varphi(X^{\mu}_s) - r\varphi(X^{\nu}_s) \mid \mathcal{W} \right] ds\right]\\
	\leq &r\int_{0}^t\mathbb{E}\vert X_s^{\mu} - X_s^{\nu} \vert ds,
	\quad \forall \mu, \nu \in L^1_{\omega, t}\mathcal{M}_r.
\end{align*}
Using standard estimates for the solutions of SDEs, Lemma \ref{lemma:conditional integrals}, the Burkholder-Davis-Gundy inequality, Assumption \ref{assumptions:Lipschitz} and Gronwall's Lemma, we have
\begin{equation*}
	\mathbb{E}\vert X^{\mu}_t - X^{\nu}_t \vert
	\leq e^{Ct} \int_{0}^t\mathbb{E}\rho(\mu_s, \nu_s)ds.
\end{equation*}
Hence,
\begin{equation*}
\mathbb{E}\int_{0}^t\rho(\Phi(\mu)_s, \Phi(\nu)_s)ds
\leq
r e^{CT}\int_{0}^t\int_{0}^s\mathbb{E}\rho(\Phi(\mu)_r, \Phi(\nu)_r)drds,
\end{equation*}
where the constant $C > 0$ depends only on $r$ and $K$ as given in Assumption \ref{assumptions:particles}.
Iterating the operator $\Phi$ $k$-times, yields the following inequality
\begin{align*}
\mathbb{E}\int_{0}^T \rho(\Phi^k(\mu)_t, \Phi^k(\nu)_t) dt
\leq & r\int_{0}^T\mathbb{E}\vert X^{\Phi^{k-1}(\mu)}_t - X^{\Phi^{k-1}(\nu)}_t \vert dt \\
\leq & r^k e^{kCT} \int_{0}^T \dots \int_{0}^{t_{k-1}}\mathbb{E}\rho(\mu_{t_k}, \nu_{t_k})dt_k \cdots dt_{1}\\
\leq & \frac{r^k e^{kCT}}{(k-1)!} \int_{0}^T\mathbb{E}\rho(\mu_t, \nu_t)dt.
\end{align*}
If $k$ is large enough, the coefficient $r^ke^{kCT} / (k-1)!$ is less then one. Hence, $\Phi^k$ is a contraction on $L^1_{\omega, t}\mathcal{M}_r$ and thus has a unique fixed point. This fixed point is also the unique fixed point of $\Phi$, see \cite[Prop 2.3]{coghi2018pathwise}. Since solutions to \eqref{sde system} are precisely the fixed points of $\Phi$, this yields the existence and uniqueness of solutions to the McKean-Vlasov equation \eqref{sde system}.
%We turn now to the proof of stability. Let $X_0^1, X_0^2 \in L^1(\Omega, \mathcal{F}_0;\mathbb{R}^d)$ be two initial conditions and $(X^1, \mu^1), (X^2, \mu^2)$ the associated solutions. Following the same reasoning as in the proof of well-posedness we obtain
%\begin{align*}
%	\mathbb{E}\left[\sup_{t\in[0,T]}\rho(\mu^1_t, \mu^2_t) \right] 
%	\leq & r \mathbb{E}\left[\sup_{t\in[0,T]}\sup_{\varphi \in \operatorname{Lip}_1}\mathbb{E}\left[\varphi(X^1_t) - \varphi(X^2_t) \mid \mathcal{W} \right]\right]\\
%	\leq & C\left( \mathbb{E}\vert X^1_0 - X^2_0 \vert + \int_{0}^{T}  \mathbb{E}\left[\sup_{s\in[0,t]}\rho(\mu^1_s, \mu^2_s)\right]dt \right),
%\end{align*}
%where $C>0$ is possibly different than before. A standard application of Gronwall's lemma is enough to conclude.
\end{proof}

\begin{rem}
	\label{rem:law indep initial condition}
	We note that, under more restrictive assumptions on the coefficients, the conditional law $\mu = \mathcal{L}(X \mid \mathcal{W})$ of a solution $X$ to \eqref{sde system} does not depend on $X_0$ but only on $\mu_0 := \mathcal{L}(X_0)$. This follows from the results proved later in Section \ref{section: main results}. Indeed,	
	%	and let $(X, \mu)$ be the solution to equation \ref{sde system} with initial condition $X_0$.
	%	Theorem \ref{thm:existence spde} shows that $\mu$ solves the nonlinear equation \ref{target equation}, which implies that $\mu$ solves the linear SPDE \ref{linear ito spde} with initial condition $\mu_0$ and coefficients given by plugging $\mu$ into $b, a$ and $\sigma$.
	%
	assuming Assumption \ref{assumptions:nonlinear} with $m >\frac{d}{2} + 2$, Theorem \ref{thm:existence spde}
	implies that $\mu$ is a solution to equation \eqref{target equation} in the sense of Definition \ref{definition of solution}. Thanks to Theorem \ref{thm:uniqueness spde}, this solution is unique, given the initial law $\mu_0$. This implies that $\mu$ only depends on $X_0$ via its law $\mu_0 = \mathcal{L}(X_0)$.
\end{rem}

\subsection{Remarks on the associated interacting particle system}
\label{section:interacting particles}

Let $(\Omega, \mathcal{F}, (\mathcal{F}_t)_{t\geq 0},\mathbb{P})$ be a filtered probability space. Let $W$ be an $(\mathcal{F}_t)_{t\geq 0}$-compatible Brownian motion and $(X^i_0)_{i\geq 0}$ be a sequence of independent and identically distributed (IID) random variables in $L^2(\Omega, \mathcal{F}_0;\mathbb{R}^d)$ with law $\mu_0$. Moreover, consider a sequence of independent $(\mathcal{F}_t)_{t\geq 0}$-adapted Brownian motions $(B^i_t)_{t\geq 0}$, which are jointly independent of $W$ and $(X^i_0)_{i\geq 0}$. 

Consider the following system of interacting particles on $\mathbb{R}^d$, $\forall t \in [0,T]$, $i = 1, \dots, N$,
\begin{equation*}
\left\{
\begin{array}{l}
dX^{i,N}_t = b(t, X^{i,N}_t, L^N_t)dt + \sigma(t, X^{i,N}_t, L^N_t)dW_t + \alpha(t, X^{i,N}_t, L^N_t)dB^i_t\\
X^{i,N}_0 = X_0^i,
\end{array}
\right.
\end{equation*}
where $L^N_t := \frac{1}{N}\sum_{i=1}^{N}\delta_{X_t^{i,N}}$ is the empirical measure of the system.

In this section, we work under the following additional assumption.
\begin{asm}
	\label{assumptions:extra particles}
	There is a constant $K>0$ such that, for any IID sequence of random variables $(X^i)_{i\geq0}$ on $\mathbb{R}^d$, with law $\mu$, the following holds, for every $x\in \mathbb{R}^d$,
	\begin{equation*}
	\mathbb{E}\left|\sigma(x, L^N) - \sigma(x, \mu)\right|^2
	+\mathbb{E}\left|b(x, L^N) - b(x, \mu)\right|^2
	+\mathbb{E}\left|\alpha(x, L^N) - \alpha(x, \mu)\right|^2
	\leq \frac{K^2}{N},
	\end{equation*}
	where $L^N:= \frac{1}{N}\sum_{i=1}^{N}\delta_{X^{i}}$.
\end{asm}

It is proved in \cite[Theorem 2.3]{MR1797090} that each particle $X_t^{i,N}$ converges to a solution $X_t^i$ of equation \eqref{sde system} with initial condition $X^i_0$ and driving noise $W$ and $B^i$, in the sense, that, for each $i\geq 0$ and $T>0$,
\begin{equation*}
\mathbb{E}\left[
\sup_{t\in[0,T]} \vert X^{i,N}_t - X^i_t \vert^2
\right] \leq \frac{C(T)}{N}.
\end{equation*}
Moreover, from \cite[Corollary 2.4]{MR1797090}, the empirical measure $L^N_t$ converges, as $N\to\infty$, to the conditional law $\mu_t = \mathcal{L}(X^1_t \mid \mathcal{W})$, that is, for each $\varphi\in \operatorname{Lip}_1$, $t\in[0,T]$,
\begin{equation}\label{particles:convergence empirical measure}
\mathbb{E}\vert 
\langle L^N_t, \varphi\rangle - \langle \mu_t, \varphi\rangle 
\vert
\leq \frac{C(t)}{\sqrt{N}}.
\end{equation}
Notice that here $X^1$ is not special, we could define $\mu_t = \mathcal{L}(X^i_t \mid \mathcal{W})$, for any $i\geq 0$, and have the same result. Moreover, we will see in Section \ref{section:lienear spde} that $\mu$ is the solution to equation \eqref{target equation} as given by Theorem \ref{thm:existence spde} below.

A result of propagation of chaos, similar to the one stated in \cite{sznitman1991topics} can be obtained. In this case, however, the propagation of chaos is conditional to the common noise $W$.
\begin{lem}
	The interacting particles $(X^{i,N})_{i=1,\dots, N}$ are $\mu$ chaotic, conditional to $W$, in the sense that, for $k\in \mathbb{N}$, and $\varphi^1, \dots, \varphi^k \in \varphi\in \operatorname{Lip}_1$, we have
	\begin{equation*}
	\lim_{N\to\infty}\mathbb{E}\left\vert
	\mathbb{E}\left[
	\varphi^1(X^{1,N}_t) \cdot \ldots \cdot \varphi^k(X^{k,N}_t) 
	\mid \mathcal{W} \right] 
	- \prod_{i=1}^{k}\langle \mu_t , \varphi^i \rangle
	\right\vert = 0,
	\quad \forall t \in[0,T].
	\end{equation*}
\end{lem}

\begin{proof}
	Without loss of generality, assume $k=2$.  First notice that, for $i \neq j$, $\forall t\in [0,T]$, $X^i_t$ is independent of $X^1_t$, conditionally to $\mathcal{W}$, which implies,
	\begin{equation*}
	\langle \mu_t , \varphi^1 \rangle\langle \mu_t , \varphi^2 \rangle = 
	\mathbb{E}\left[
	\varphi^1(X_t^i)\varphi^2(X_t^j)
	\mid \mathcal{W}\right]
	\quad \mathbb{P}-a.s.,
	\end{equation*}
	where we used that, for every $i\geq 1$, $\mu_t = \mathcal{L}(X^i_t \mid \mathcal{W})$.
	Moreover, the particles are exchangeable, even when conditioned to $\mathcal{W}$, $	\forall i\neq j, \forall t \in [0,T]$,
	\begin{equation*}
	\mathcal{L}((X_t^{i,N}, X_t^{j,N}) \mid \mathcal{W}) = \mathcal{L}((X_t^{1,N}, X_t^{2,N}) \mid \mathcal{W}),
	\quad
	\mathcal{L}((X_t^{i,N}, X_t^{i}) \mid \mathcal{W}) = \mathcal{L}((X_t^{1,N}, X_t^{1})).
	\end{equation*}
	We observe, $\forall t \in [0,T]$, $\mathbb{P}$-a.s.,
	\begin{align}
	& 
	\mathbb{E}\left[
	\varphi^1(X^{1,N}_t) \varphi^2(X^{2,N}_t) 
	\mid \mathcal{W} \right] 
	- \langle \mu_t , \varphi^1 \rangle\langle \mu_t , \varphi^2 \rangle
	\nonumber \\
	& = \mathbb{E}\left[\frac{1}{N(N-1)}\sum_{i\neq j = 1}^{N}\left[
	\varphi^1(X_t^{i,N})\left(\varphi^2(X_t^{j,N}) - \varphi^2(X^j_t)\right)
	\right]\mid\mathcal{W}\right] \label{particles:term to estimate}\\
	& + \mathbb{E}\left[\frac{1}{N(N-1)}\sum_{i\neq j = 1}^{N}\left[
	\varphi^2(X_t^{i})\left(\varphi^1(X_t^{j,N}) - \varphi^1(X_t^j)\right)
	\right]\mid\mathcal{W}\right].\nonumber
	\end{align}
	We take the absolute value and expectation and show that both terms on the right hand side converge to zero as $N \to \infty$. We only consider the first term on the right hand side of \eqref{particles:term to estimate}, since the treatment of the remaining one proceed analogously.
	\begin{align*}
	\mathbb{E}&\left \vert 
	\mathbb{E}\left[\frac{1}{N(N-1)}\sum_{i\neq j = 1}^{N}\left[
	\varphi^1(X_t^{i,N})\left(\varphi^2(X_t^{j,N}) - \varphi^2(X_t^j)\right)
	\right]\mid\mathcal{W}\right]
	\right\vert \\
	\leq &\mathbb{E}\left \vert \frac{1}{N(N-1)} \sum_{i=1}^{N}\varphi^1(X_t^{i,N}) \sum_{j\neq i} \left(
	\varphi^2(X_t^{j,N}) - \varphi^2(X^j_t)
	\right)\right\vert\\
	\leq & \Vert \varphi^1 \Vert_{C^0} \mathbb{E}\left \vert \frac{1}{(N-1)} \sum_{j\neq 1} \left(
	\varphi^2(X_t^{j,N}) - \varphi^2(X_t^j)
	\right)\right\vert\\
	\leq & \Vert \varphi^1 \Vert_{C^0} \frac{N}{N-1} \mathbb{E} \left \vert
	\frac1N\sum_{j=1}^N\varphi^2(X_t^{i,N}) - \langle \mu_t, \varphi^2 \rangle 
	\right\vert
	+ \Vert \varphi^1 \Vert_{C^0}\frac{1}{N-1} \mathbb{E}\vert \varphi^2(X_t^{1,N}) \vert.
	\end{align*}
	This last quantity goes to zero because of \eqref{particles:convergence empirical measure}.
\end{proof}

\section{Linear stochastic pde}
\label{section:lienear spde}
Let $(\Omega, \mathcal{F},( \mathcal{F}_t)_{t\geq 0}, \mathbb{P})$ be a filtered probability space, compatible with a $d_1$-dimensional Brownian motion $W$ in the sense of Definition \ref{defn:filtration}.

In this section we study the well-posedness of solutions to the linear version of \eqref{target equation}, that is, the SPDE,
\begin{equation}
\label{linear ito spde}
	\left\{
\begin{array}{l}
\partial_t \mu
= \partial^2_{i,j}(a^{ij} \mu)
- \partial_{i} (b^i \mu)  
- \partial_{i}(\sigma^{ik}\mu) dW^k_t,\\
\mu|_{t=0} = \mu_0 \in \mathcal{M}_r(\mathbb{R}^d),
\end{array}
\right.
\end{equation}

where $a(t, x, \omega), \sigma(t, x, \omega)$ and $b(t, x, \omega)$ satisfy the following assumptions
\begin{asm}
	\label{assumptions:linear}
	Let $m\in \mathbb{N} \cup \{0\}$.
	\begin{enumerate}[label=(\roman*), ref=\ref{assumptions:linear} (\roman*)]
		%[label=\textbf{B.\arabic*}]
		\item \label{assumptions:linear a}The function $a(t,x,\omega) := (a^{i,j}(t,x,\omega)):\mathbb{R}_{+}\times \mathbb{R}^d\times \Omega \to \mathbb{S}^d$ is measurable and $(\mathcal{W}_t)_{t \geq 0}$-adapted. Moreover, there exists a positive constant $K_m$ such that for all $(t,\omega) \in \mathbb{R}_{+}\times \Omega$, $a(t,\cdot,\omega) \in C^m(\mathbb{R}^d;\mathbb{S}^d)$ (the set of $m$-times bounded differentiable functions on the space of real symmetric $d\times d$ matrices) and
		\begin{equation*}
		\sup_{t,\omega}\| a(t,\cdot,\omega)\|_{C^m} \leq K_m.
		\end{equation*}
		
		\item \label{assumptions:linear b} The function $b(t,x,\omega) := (b^i(t,x,\omega)):\mathbb{R}_{+}\times \mathbb{R}^d\times \Omega \to \mathbb{R}^d$ is measurable and $(\mathcal{W}_t)_{t \geq 0}$-adapted. Moreover, there exists a positive constant $K_m$ such that for all $(t,\omega) \in \mathbb{R}_{+}\times \Omega$, $b(t,\cdot,\omega) \in C^m(\mathbb{R}^d;\mathbb{R}^d)$ and
		\begin{equation*}
		\sup_{t,\omega}\| b(t,\cdot,\omega)\|_{C^m} \leq K_m.
		\end{equation*}
		
		\item \label{assumptions:linear sigma} The function $\sigma(t,x,\omega) := (\sigma^{i,k}(t,x,\omega)):\mathbb{R}_{+}\times \mathbb{R}^d\times \Omega \to \mathbb{R}^{d \times d_1}$ is measurable and $(\mathcal{W}_t)_{t \geq 0}$-adapted. Moreover, there exists a positive constant $K_m$ such that for all $(t,\omega) \in \mathbb{R}_{+}\times \Omega$, $\sigma(t,\cdot,\omega) \in C^m(\mathbb{R}^d;\mathbb{R}^{d \times d_1})$ and
		\begin{equation*}
		\sup_{t,\omega}\| \sigma(t,\cdot,\omega)\|_{C^m} \leq K_m.
		\end{equation*}
			
		\item \label{assumptions:linear Lipschitz} (Uniform Lipschitz continuity) There exists a constant $K > 0$ such that, for $t\in [0,T]$, $x, x^\prime \in \mathbb{R}^d$ and $\omega \in \Omega$,
		\begin{align*}
		\Vert & a(t,x,\omega) - a(t,x^\prime,\omega) \Vert 
		+ \Vert \sigma(t,x,\omega) - \sigma(t,x^\prime,\omega) \Vert
		+ \vert b(t,x,\omega) - b(t,x^\prime,\omega) \vert 
		 \leq K \vert x - x^\prime \vert.
		\end{align*}			
		\item \label{assumptions:linear parabolicity} (Parabolicity) For each $(t,x,\omega) \in [0, T] \times \mathbb{R}^d \times \Omega $,
		\begin{equation*}
		[2a^{ij}(t,x,\omega) - \sigma^{ik}\sigma^{jk}(t,x,\omega)] \xi^i\xi^j \geq 0, \quad \forall \xi \in \mathbb{R}^d.
		\end{equation*}	
	\end{enumerate}
\end{asm}

\begin{rem}
	Assumption \ref{assumptions:linear Lipschitz} is implied by Assumptions (i)-(iii), if $m \geq 1$.
\end{rem}

In the following we fix $r>0$ and we assume that Assumption \ref{assumptions:linear} is satisfied with $m=0$.

\begin{defn}\label{definition of linear solution}
	We say that $\mu \in L^1_{\omega, t}\mathcal{M}_r$ is a solution to equation \eqref{linear ito spde} with initial condition $\mu_0 \in \mathcal{M}_r(\mathbb{R}^d)$, if for every $\varphi \in C^2(\mathbb{R}^d)$ and $t\in[0,T]$ there exists a set of full measure $\Omega^\prime \subset \Omega$ on which the following integral equation is satisfied,
	\begin{equation}
	\label{linear integral equation}
	\langle \mu_t, \varphi \rangle 
	=  \langle \mu_0, \varphi \rangle
	+ \int_{0}^{t} \langle \mu_s, a^{i,j} \partial_{i,j}^2 \varphi\rangle ds
	+ \int_{0}^{t} \langle \mu_s, b^i \partial_i \varphi \rangle ds
	+ \int_{0}^{t} \langle \mu_s, \sigma^{i,k} \partial_i \varphi \rangle d W^k_s.
	\end{equation}
\end{defn}

\begin{rem}
	\label{remark on adaptedness}
	We note that that all the terms in the right-hand side of \eqref{linear integral equation} are well-defined, because the coefficients $a, b, \sigma$ are $(\mathcal{W}_t)_{t\geq 0}$-adapted and bounded and each $(\mathcal{B}(\mathbb{R}_{+})\times \mathcal{F})$-measurable, $(\mathcal{W}_{t})_{t\geq 0}$-adapted process has a predictable $dt \otimes \mathbb{P}$-version \cite[Theorem~3.8]{MR3136102}.
\end{rem}

We next consider the existence of solutions to the linear equation \eqref{linear ito spde}. Consider the linear version of system \eqref{sde system}, that is
\begin{equation}\label{linear sde system}
\left\{
\begin{array}{l}
dX_t = b(t,X_t)dt + \sigma(t,X_t) dW_t + \alpha(t,X_t)dB_t\\
X_t|_{t=0} = X_0,\\
\end{array}
\right.
\end{equation}
where the coefficients are fixed, $(\mathcal{F}_{t})_{t\geq 0}$-adapted stochastic processes satisfying Assumptions \ref{assumptions:linear} with $m=0$. $X_0:\Omega \rightarrow \mathbb{R}^d$ is an $\mathcal{F}_0$-measurable random variable.
As we noted in the proof of Theorem \ref{wellposedness sde}, there exists a unique time-continuous $(\mathcal{F}_t)_{t\geq 0}$-adapted solution $(X_t)_{t\in [0,T]}$ to equation \eqref{linear sde}, see \cite[Theorem 3.1.1]{MR3410409}.
Using this solution, the following lemma can be proved in the same way as Theorem \ref{thm:existence spde}, below.
\begin{lem}
	\label{existence linear spde}
	Let $X_0:\Omega \rightarrow \mathbb{R}^d$ be an $\mathcal{F}_0$-measurable random variable, $r>0$ and define $\mu_0 := r\mathcal{L}(X_0)$. Assume Assumption \ref{assumptions:linear} with $m=0$. Then, there exists a solution $\mu \in L^1_{\omega}C_t\mathcal{M}_r$ to equation \eqref{linear ito spde} in the sense of Definition \ref{definition of linear solution}.
	This solution is given as $\mu = (\mu_t)_{t\in [0,T]} := (r \mathcal{L}(X_t \mid \mathcal{W}))_{t\in [0,T]}$ for every $t \in [0,T]$, where $X_t$ is the unique strong solution to equation \eqref{linear sde system}.
\end{lem}

To prove the uniqueness of solutions to the linear equation \eqref{linear ito spde}, we introduce the dual BSPDE. We fix $t \in [0, T]$ and we take a test function $\varphi$  which is $\mathcal{F}_t \times \mathcal{B}(\mathbb{R}^d)$-measurable with $\varphi \in L^\infty(\Omega, \mathcal{F}_t, C_c^\infty(\mathbb{R}^d))$. Consider the following BSPDE
\begin{equation}\label{bspde}
\left\{
\begin{array}{l}
\partial_t f 
= - a^{i,j} \partial_{i,j}^2 f
- b^i \partial_{i} f 
- \sigma^{i,k} \partial_{i} v^k
+ v^k dW^k \\
f_t = \varphi,
\end{array}
\right.
\end{equation}
on $\mathbb{R}^d\times [0, t]$.
\begin{defn}
	\label{defn: back sol}
	 Let $t\in[0,T]$, $m\geq 1$ and  $\varphi \in L^\infty(\Omega, \mathcal{W}_t, C_c^\infty(\mathbb{R}^d))$. A process $(f,v) : [0,t] \times \Omega \to W^{2,m}  \times W^{1,m} $ is a solution to equation \eqref{bspde} with terminal condition $\varphi$, if it is progressively measurable and if there is a set $\Omega^\prime \subseteq \Omega$ of full measure such that
	\begin{equation*}
	f(s, x) = \varphi(x)  
	+ \int_s^t \left( a^{i,j} \partial_{i,j}^2 f
	+ b^i  \partial_{i} f 
	+ \sigma^{i,k} \partial_{i} v^k
	\right)(r,x) dr
	- \int_s^t v^k(r,x) dW_r^k \\
	\end{equation*}
	for all $\omega \in \Omega^\prime$, $s \in [0 , t]$ and $x \in \mathbb{R}^d$.
\end{defn}
Let $m$ be an integer such that $2(m-2) > d$. If the coefficients $b, \sigma,a$ satisfy Assumptions \ref{assumptions:linear}, it follows from \cite[Theorem 2.1, Corollary 2.2]{MR3017034} that there exists a pair of random fields $(f, v)$ such that,
\begin{equation}
\label{regularity of f and v}
f\in L^2_\omega C^0_tW^{m, 2}_x 
\quad \mbox{and } \sigma \cdot \nabla f + v \in L^2_{\omega,t} W^{m, 2}_x,
\end{equation}
which is jointly continuous in $(t,x)$ and is a strong solution to equation \eqref{bspde} in the sense of Definition \ref{defn: back sol}.

By the assumptions on $m$ and by the Sobolev embedding theorem it follows that 
\begin{equation}
\label{more regularity of f and v}
f, \sigma \cdot \nabla f + v \in L^2_\omega C^0_tC_x^2(\mathbb{R}^2),
\quad 
\mbox{which implies that}
\quad
v \in L^2_{\omega, t}C_x^1(\mathbb{R}^d). 
\end{equation}

We can now show the duality between equations \eqref{linear ito spde} and \eqref{bspde}.

%\begin{lem}
%	\label{duality lemma}
%	Fix $t\in [0,T]$ and let $m$ be an integer such that $2(m-2) > d$. If the coefficients $b, \sigma, a$ are $(\mathcal{F}_t)_{t\in [0,T]}$-adapted processes satisfying Assumptions \ref{assumptions:linear}, and if $\mu \in L^1_{\omega,t}\mathcal{M}$ satisfies the integral equation \eqref{linear integral equation} and $f$ is a solution to equation \eqref{bspde} in the sense of Definition \ref{defn: back sol} with terminal condition $f_t = \varphi \in L^\infty(\Omega, \mathcal{W}_t, C_c^\infty(\mathbb{R}^d))$, then
%	\begin{equation*}
%	\mathbb{E}\langle \mu_t, \varphi \rangle = \mathbb{E}\langle \mu_0, f_0 \rangle.
%	\end{equation*}
%\end{lem}

\begin{lem}
	\label{duality lemma}
	Fix $t\in [0,T]$, let $m$ be an integer such that $2(m-2) > d$ and let $b, \sigma, a$ be $(\mathcal{W}_t)_{t\in [0,T]}$-adapted processes satisfying Assumptions \ref{assumptions:linear}. 
	
	Let $\mu : [0,T]\times \Omega \rightarrow \mathcal{M}_{\leq 2r}^{\pm}$ be an $(\mathcal{W}_t)_{t\in [0,T]}$-adapted process, such that for every $\varphi \in C^2(\mathbb{R}^d)$ and $t\in[0,T]$ there exists a set of full measure $\Omega^\prime \subset \Omega$ on which $\mu$ satisfies the integral equation \eqref{linear integral equation}. 
	
	If $f$ is a solution to equation \eqref{bspde} in the sense of Definition \ref{defn: back sol} with terminal condition $f_t = \varphi \in L^\infty(\Omega, \mathcal{W}_t, C_c^\infty(\mathbb{R}^d))$, then
	\begin{equation*}
	\mathbb{E}\langle \mu_t, \varphi \rangle = \mathbb{E}\langle \mu_0, f_0 \rangle.
	\end{equation*}
\end{lem}

\begin{proof}
	Let $\eta_\epsilon$ be a standard mollifier, i.e., $\eta_\epsilon(x) := \frac{1}{\epsilon^d}\eta(\frac{x}{\epsilon})$, with
	\begin{equation*}
	\eta \in C_c^\infty(\mathbb{R}^d), \quad \eta \geq 0, \quad \int_{\mathbb{R}^d} \eta(x) dx = 1.
	\end{equation*}
	 For each $x \in \mathbb{R}^d$, we define $\eta^\epsilon_x (y) := \eta^\epsilon(x - y)$. The function $\mu^\epsilon := \mu \ast \eta^\epsilon$ satisfies: $\forall x \in \mathbb{R}^d, \forall \epsilon > 0$, there exists a set of full measure $\Omega_{x,\epsilon}$, such that
	\begin{align*}
	\mu_t^\epsilon(x)
	= & \mu_0^\epsilon(x)
	- \int_{0}^{t} \langle \mu_s, \sigma_s^{i,k} \partial_{x_i} \eta^\epsilon_x \rangle d W^k_s
	+ \int_{0}^{t} \langle \mu_s, a_s^{i,j} \partial_{x_i,x_j}^2 \eta^\epsilon_x - b_s^i \partial_{x_i} \eta^\epsilon_x\rangle ds,
	\quad \forall\omega \in \Omega_{x,\epsilon}.
	\end{align*}
	
	Using It\^{o}'s formula we can compute the product $\mu^\epsilon_t(x) f_t(x) $ and obtain the following equality on a set of full measure $\Omega_{x,\epsilon}$, possibly different from the previous one,
	\begin{align}
	\mu^\epsilon_t(x)f_t(x)= & \mu_0^\epsilon(x)f_0(x)
	- \int_{0}^{t} \langle \mu_s, \sigma_s^{i,k} \partial_{x_i} \eta^\epsilon_x \rangle f_s(x) d W^k_s
	+ \int_{0}^{t} \mu_s^\epsilon(x) v^k_s(x) d W^k_s \label{martingales in duality}\\
	& + \int_{0}^{t}  \langle \mu_s, a_s^{i,j} \partial_{{x_i},{x_j}}^2 \eta^\epsilon_x - b_s^i \partial_{x_i} \eta^\epsilon_x\rangle f_s(x) ds \label{f terms in duality_0}\\
	& - \int_{0}^{t} \mu_s^\epsilon(x) \left[ a_s^{i,j} \partial_{{x_i},{x_j}}^2f_s + b^i_s \partial_{x_i} f_s \right](x) ds\label{f terms in duality}\\
	& - \int_{0}^{t}  \langle \mu_s , \sigma^{i,k} \partial_{x_i} \eta_x^\epsilon \rangle v^k_s(x) ds
	- \int_{0}^{t} \mu^\epsilon_s(x) [\sigma^{i,k} \partial_{x_i} v^k_s](x) ds. \label{v terms in duality}
	\end{align}
	
	For every $x$, we can take the expectation of both sides of \eqref{martingales in duality} - \eqref{v terms in duality} and obtain an equality on all of $\mathbb{R}^d$. We notice that the It\^o integrals in \eqref{martingales in duality} are martingales because their arguments are in $L^2_{\omega, t}$, which is a consequence of \eqref{more regularity of f and v}. It follows that both martingale terms vanish in expectation.
	
	Next, we integrate over $\mathbb{R}^d$, use Fubini's theorem to interchange Lebesgue integration and expectation, and take the limit $\epsilon\to 0$. It remains to identify the limit as $\epsilon\to 0$ of each of the resulting terms:
	\begin{align*}
	\mathbb{E}\langle \mu_t^\epsilon, f_t\rangle 
	- \mathbb{E}\langle \mu_t, f_t \rangle 
	&= \mathbb{E}\left[\iint_{\mathbb{R}^d\times \mathbb{R}^d} \eta^\epsilon(x-y)f_t(x)d\mu_t(y)dx
	- \int_{\mathbb{R}^d} f_t(x)d\mu_t(x)\right] \\
	& =\mathbb{E}\int_{\mathbb{R}^d} [(\eta^\epsilon \ast f_t)(x) - f_t(x) ]d \mu_t(x).
	\end{align*}
	
	From the regularity of $f$ and $v$, namely \eqref{more regularity of f and v}, it follows that the integration of $f_t(x)$ with respect to $\mu(dx)$ is well defined. The joint time-space continuity of $f$ together with the maximum principle for the solution of the backward equation \eqref{bspde}, which is proved in \cite[Corollary 2.3]{MR3017034}, imply that there exists a constant $C > 0$, such that, for all $t\in[0,T]$, $x\in \mathbb{R}^d$ and for almost all $\omega \in \Omega$,
	\begin{equation*}
	\vert f_t(x,\omega) \vert \leq C.
	\end{equation*}
	Hence, we can apply dominated convergence to conclude 
	\begin{equation*}
	\mathbb{E}\int_{\mathbb{R}^d} [(\eta^\epsilon \ast f_t)(x) - f_t(x) ]d \mu_t(x) \to 0, \quad \mbox{as } \epsilon \to 0.
	\end{equation*}
	The same argument can be applied at time $t=0$. We next study the convergence of \eqref{v terms in duality}.
	\begin{align}
	-\mathbb{E}\int_{0}^{t}& \langle \langle \mu_s , \sigma^{i,k} \partial_{x_i} \eta_\cdot^\epsilon \rangle, v^k_s \rangle ds
	- \mathbb{E}\int_{0}^{t} \langle \mu^\epsilon_s, \sigma^{i,k} \partial_{x_i} v^k_s\rangle ds \nonumber\\
	= & \mathbb{E}\int_{0}^{t}\int_{\mathbb{R}^d} \sigma^{i,k} (\partial_i v^k_s \ast \eta^\epsilon)(x)\mu_s(dx)ds
	- \mathbb{E}\int_{0}^{t}\int_{\mathbb{R}^d} [(\sigma^{i,k} \partial_{i}v^k_s) \ast \eta^\epsilon](x) \mu_s(dx)ds. \label{opposite terms}
	\end{align}
	Both the first and the second integrand converge, as $\epsilon \to 0$ to $\sigma_s^{i,k}(x) \partial_{i} v_s^k(x)$ because of the properties of the mollifier $\eta_\epsilon$. We aim to take the limit under the integration to conclude that the right-hand side of \eqref{opposite terms} converges to zero in the limit $\epsilon \downarrow 0$. Notice that
	\begin{equation*}
	\sigma^{i,k} (\partial_i v^k_s \ast \eta^\epsilon)(x), \; [(\sigma^{i,k} \partial_{i}v^k_s) \ast \eta^\epsilon](x) \leq K_{m}\Vert v_s \Vert_{C^1},
	\quad \mathbb{P}-a.s.,
	\end{equation*}
	which is integrable with respect to the measure $\mu_s(dx)ds\mathbb{P}(d\omega)$ since $v \in L^2_{\omega, t}C_x^1$. Hence, we can apply the dominated convergence theorem in \eqref{opposite terms} to conclude.
	
	The same reasoning can be applied to show that \eqref{f terms in duality_0}, \eqref{f terms in duality} converge to
	\begin{equation*}
	\pm\int_{0}^{t} \langle \mu_s, a_s^{i,j} \partial_{i,j}^2 f_s + b_s^i \partial_i f_s \rangle ds,
	\end{equation*}
	and thus cancel in the limit $\epsilon \downarrow 0$.
\end{proof}

\begin{theorem}\label{wellposedness linear spde}
	Let $m$ be an integer such that $2(m-2) > d$. If the coefficients $b, \sigma, a$ are $(\mathcal{F}_t)_{t\in[0,T]}$-adapted processes satisfying Assumptions \ref{assumptions:linear}, then, equation \eqref{linear ito spde} admits at most one solution in the sense of Definition \ref{definition of linear solution}.
\end{theorem}

\begin{proof}
	We use the weak formulation \eqref{linear integral equation} and a duality arguments to prove that the solution to equation \eqref{linear ito spde} is unique if equation \eqref{bspde} admits a classical solution. Let $\mu^1, \mu^2 \in L^1_{\omega, t}\mathcal{M}_r$ be two solutions to \eqref{linear ito spde} in the sense of Definition \ref{definition of linear solution} and define $\mu := \mu^1 - \mu^2: [0,T] \times \Omega \rightarrow \mathcal{M}^{\pm}_{\leq 2r}$. Let $t \in [0,T]$, $\varphi \in C_c^\infty(\mathbb{R}^d)$, and define $\tilde \varphi := \operatorname{sign}\langle \mu_t, \varphi \rangle \varphi \in L^\infty(\Omega, \mathcal{W}_t, C_c^\infty(\mathbb{R}^d))$.
	
	By linearity, $\mu$ satisfies the integral formulation \eqref{linear integral equation}. Hence, it follows from Lemma \ref{duality lemma}, applied to equation \eqref{bspde} with terminal condition $\tilde \varphi$, that
	\begin{equation*}
	\mathbb{E}\vert \langle \mu_t, \varphi \rangle \vert  = \mathbb{E}[\langle \mu_t, \tilde \varphi \rangle] = \mathbb{E}[\langle \mu_0, f_0 \rangle] = 0,
	\end{equation*}
	where the last equality is satisfied when the initial conditions, $\mu^1_0$ and $\mu^2_0$, are the same. Hence, we have that $\langle \mu_t, \varphi \rangle = 0$ on a set of full measure. Since the function $t \mapsto \langle \mu_t, \varphi \rangle$ is continuous, we can easily see that the set of full measure only depends on $\varphi$, that is, 
	\begin{equation*}
	\forall \varphi \in C_c^\infty(\mathbb{R}^d, \mathbb{R}_{+}), 
	\quad \exists \Omega_\varphi \subset \Omega, 
	\quad \forall t \in [0, T], 
	\quad \langle \mu_t, \varphi \rangle = 0.
	\end{equation*} 
	Let $R \in \mathbb{N}$, we call $B_R \subset \mathbb{R}^d$, the ball of radius $R$. We have that the space $C_c^\infty(B_R, \mathbb{R}_{+})$ is separable. Hence,
	\begin{equation*}
	\forall R \in \mathbb{N},
	\quad \exists \Omega_R \subset \Omega,
	\quad \forall \varphi \in C_c^\infty(B_R, \mathbb{R}_{+}), 
	\quad \forall t \in [0, T],
	\quad \langle \mu_t, \varphi \rangle = 0.
	\end{equation*} 
	We conclude the proof by noticing that on the set of full measure $\cap_{R\in \mathbb{N}} \Omega_R $, $\langle \mu_t, \varphi \rangle = 0$, $\forall \varphi \in C_c^\infty(\mathbb{R}^d, \mathbb{R}_{+})$, $\forall t \in [0, T]$.
\end{proof}

\section{Non-local Fokker-Planck equations}
\label{section: main results}
In this section we study the existence and uniqueness of solutions to equation \eqref{target equation}. After giving the definition of a solution, we prove the existence via the McKean-Vlasov SDE \eqref{sde system}. The uniqueness follows from the uniqueness of solutions to the linear equation and well-posedness of the McKean-Vlasov SDE.

We will use the following assumptions.

\begin{asm}
	\label{assumptions:nonlinear}
	Let $m \in \mathbb{N} \cup \{0\}$, assume
	\begin{enumerate}[label=(\roman*), ref=\ref{assumptions:nonlinear} (\roman*)]
		%[label=\textbf{C.\arabic*}]
		
		\item The function $a(t,x,\mu) := (a^i(t,x,\mu)):\mathbb{R}_{+}\times \mathbb{R}^d\times \mathcal{M}(\mathbb{R}^d) \to \mathbb{S}^d$ is measurable. Moreover, there exists a positive constant $K_m$ such that for all $(t,\mu) \in \mathbb{R}_{+}\times \mathcal{M}(\mathbb{R}^d) $, $a(t,\cdot,\mu) \in C^m(\mathbb{R}^d;\mathbb{S}^d)$ and
		\begin{equation*}
		\sup_{t,\mu}\| a(t,\cdot,\mu)\|_{C^m} \leq K_m.
		\end{equation*}
		
		\item The function $b(t,x,\mu) := (b^i(t,x,\mu)):\mathbb{R}_{+}\times \mathbb{R}^d\times \mathcal{M}(\mathbb{R}^d) \to \mathbb{R}^{d}$ is measurable. Moreover, there exists a positive constant $K_m$ such that for all $(t,\mu) \in \mathbb{R}_{+}\times \mathcal{M}(\mathbb{R}^d) $, $b(t,\cdot,\mu) \in C^m(\mathbb{R}^d;\mathbb{R}^{d \times d_1})$ and
		\begin{equation*}
		\sup_{t,\mu}\| b(t,\cdot,\mu)\|_{C^m} \leq K_m.
		\end{equation*}
		
		\item\label{assumption:sigma nonlinear} The function $\sigma(t,x,\mu) := (\sigma^{i,k}(t,x,\mu)):\mathbb{R}_{+}\times \mathbb{R}^d\times \mathcal{M}(\mathbb{R}^d) \to \mathbb{R}^{d \times d_1}$ is measurable. Moreover, there exists a positive constant $K_m$ such that for all $(t,\mu) \in \mathbb{R}_{+}\times \mathcal{M}(\mathbb{R}^d) $, $\sigma(t,\cdot,\mu) \in C^m(\mathbb{R}^d;\mathbb{R}^{d \times d_1})$ and
		\begin{equation*}
		\sup_{t,\mu}\| \sigma(t,\cdot,\mu)\|_{C^m} \leq K_m.
		\end{equation*}
		
		\item \label{assumptions:nonlinear Lipschitz} (Uniform Lipschitz continuity) There exists a constant $K > 0$ such that, for each $t \in [0, T]$ and $(x,\mu^\prime), (x,\mu^\prime)\in \mathbb{R}^d \times \mathcal{M}_r(\mathbb{R}^d)$,
		\begin{align*}
		\Vert & a(t,x,\mu) - a(t,x^\prime,\mu^\prime) \Vert 
		+ \Vert \sigma(t,x,\mu) - \sigma(t,x^\prime,\mu^\prime) \Vert
		+ \vert b(t,x,\mu) - b(t,x^\prime,\mu^\prime) \vert \\
		& \leq K\left(  \vert x - x^\prime \vert + \rho(\mu, \mu^\prime)\right),
		\end{align*}
		
		for all $\mu, \mu^\prime \in \mathcal{M}_r(\mathbb{R}^d)$, $t\in [0,T]$ and $x, x^\prime \in \mathbb{R}^d$.
		
		\item \label{assumptions:nonlinear parabolicity} (Parabolicity) For each $(t,x,\mu) \in [0, T] \times \mathbb{R}^d \times \mathcal{M}_r(\mathbb{R}^d) $,
		\begin{equation*}
		[2a^{ij}(t,x,\mu) - \sigma^{ik}\sigma^{jk}(t,x,\mu)] \xi^i\xi^j \geq 0, \quad \forall \xi \in \mathbb{R}^d.
		\end{equation*}
	\end{enumerate}
\end{asm}

Let $(\Omega, \mathcal{F},( \mathcal{F}_t)_{t\geq 0}, \mathbb{P})$ be a filtered atomless probability space, compatible with a $d_1$-dimensional Brownian motion $W$  in the sense of Definition \ref{defn:filtration}.

Assume Assumption \ref{assumptions:nonlinear} with any $m = 0$.
\begin{defn}\label{definition of solution}
	We say that $\mu \in L^1_{\omega, t}\mathcal{M}_r$ is a solution to equation \eqref{target equation} with initial condition $\mu_0 \in \mathcal{M}_r(\mathbb{R}^d)$, if for every $\varphi \in C^2(\mathbb{R}^d)$ and $t\in[0,T]$, there exists a set of full measure $\Omega^\prime \subset \Omega$ on which the following integral equation is satisfied,
	\begin{align}
	\langle \mu_t, \varphi \rangle 
	= & \langle \mu_0, \varphi \rangle
	+ \int_{0}^{t} \langle \mu_s, a^{i,j}(s,\mu_s) \partial_{i,j}^2 \varphi\rangle ds
	+ \int_{0}^{t} \langle \mu_s, b^i(s,\mu_s) \partial_i \varphi \rangle ds\nonumber\\
	& + \int_{0}^{t} \langle \mu_s, \sigma^{i,k}(s,\mu_s) \partial_i \varphi \rangle d W^k_s. \label{ito weak formulation}
	\end{align}
\end{defn}

Under Assumption \ref{assumptions:nonlinear} all the integrals in the previous definition are well defined. Moreover, $\langle \mu_s, \sigma^{i,k}(t,\mu_s) \partial_i \varphi \rangle$ is $(\mathcal{W}_{t})_{t\geq0}$-adapted and $(\mathcal{B}(\mathbb{R}_{+})\times \mathcal{F})$-measurable which is enough for the It\^{o} integral to be defined (see Remark \ref{remark on adaptedness}).

\begin{theorem}\label{thm:existence spde}
	Let $r>0$ and $\mu_0 \in \mathcal{M}_r(\mathbb{R}^d)$. If Assumptions \ref{assumptions:nonlinear} are in force with $m=0$ and if $(X, \mu)$ is a solution to equation \eqref{sde system} in the sense of Definition \ref{definition solution sde} with any initial condition $X_0 :\Omega\rightarrow \mathbb{R}^d$ such that $\mu_0 = r\mathcal{L}(X_0)$, then $\mu =r\mathcal{L}(X \mid \mathcal{W})\in L^1_{\omega}C_t\mathcal{M}_r \subset L^1_{\omega, t}\mathcal{M}_r$ is a solution to equation \eqref{target equation} in the sense of Definition \ref{definition of solution}, with initial condition $\mu_0$.	
\end{theorem}
%\begin{theorem}\label{thm:existence spde}
%	Let $r>0$ and $\mu_0 \in \mathcal{M}_r(\mathbb{R}^d)$. If Assumptions \ref{assumptions:nonlinear} are in force with $m=0$, then there exists a solution $\mu \in L^1_{\omega}C_t\mathcal{M}_r \subset L^1_{\omega, t}\mathcal{M}_r$ to equation \eqref{target equation} in the sense of Definition \ref{definition of solution}, with initial condition $\mu_0$.	
%	
%	Moreover, $\mu$ is the second component of the solution to \eqref{sde system} given by Theorem \ref{wellposedness sde}, with any initial condition $X_0 \in L^1(\Omega, \mathcal{F}_0; \mathbb{R}^d)$ such that $\mu_0 = r\mathcal{L}(X_0)$.
%	%Moreover, if $\mu_0 \in \mathcal{P}(\mathbb{R}^d)$, then $\mu\in L^1_{\omega, t}\mathcal{P}$.
%\end{theorem}
\begin{proof}
	Since the probability space $(\Omega, \mathcal{F}_0, \mathbb{P})$ is atomless, there exists an $\mathcal{F}_0$-measurable random variable $X_0 :\Omega \rightarrow \mathbb{R}^d$ such that $\mu_0 = r\mathcal{L}(X_0)$. 
	Assumption \ref{assumptions:nonlinear} imply Assumption \ref{assumptions:particles}, we can apply Theorem \ref{wellposedness sde} to get a solution $(X_t, \mu_t)$ to equation \eqref{sde system} with initial condition $X_0$. Using It\^{o}'s formula, we check that $\mu_t$ solves equation \eqref{target equation} in a distributional sense. We have
\begin{align*}
	\varphi(X_t) = &\varphi(X_0) 
	+ \int_{0}^{t} b^i(s,X_s, \mu_s) \partial_i\varphi(X_s) ds\\
	& + \int_{0}^{t} \frac12 \left[\alpha^{i,l}\alpha^{l,j} + \sigma^{i,k}\sigma^{k,j}\right](s,X_s, \mu_s)\partial^2_{i,j}\varphi(X_s) ds\\ 
	& + \int_{0}^{t} \sigma^{i,k}(s,X_s, \mu_s) \partial_i \varphi(X_s)dW^k_s\\
	& + \int_{0}^{t} \alpha^{i,j}(s,X_s, \mu_s) \partial_i \varphi(X_s) dB^j_s, \quad \mathbb{P}-a.s.
\end{align*}
By multiplying by $r$ and taking the conditional expectation with respect to $W$, we obtain equation \eqref{ito weak formulation}. This follows from the definition of $\alpha$ and Lemma \ref{lemma:conditional integrals}.

It follows from the definition of $\mu_t$ as solution of the McKean-Vlasov SDE and Theorem \ref{wellposedness sde} that $\mu \in L^1_{\omega}C_t\mathcal{M}_r$.
\end{proof}

We are ready to state the uniqueness result in the nonlinear case.

\begin{theorem}\label{thm:uniqueness spde}
	Let $r>0$ and $\mu_0 \in \mathcal{M}_r(\mathbb{R}^d)$.
	Let Assumptions \ref{assumptions:nonlinear} be satisfied with
	\begin{equation*}
	m > \frac{d}{2} + 2.
	\end{equation*}
	Then, the solution $\mu$ of equation \eqref{target equation} in the sense of Definition \ref{definition of solution} is unique and it is given by $\mu = (\mu_t)_{t\in [0,T]} := (r \mathcal{L}(X_t \mid \mathcal{W}))_{t\in [0,T]}$, where $X_t$ is a solution to equation \eqref{sde system} with initial condition $X_0 :\Omega \rightarrow \mathbb{R}^d)$ such that $\mu_0 = r\mathcal{L}(X_0)$.
\end{theorem}
%\begin{theorem}\label{thm:uniqueness spde}
%	Let $r>0$ and $\mu_0 \in \mathcal{M}_r(\mathbb{R}^d)$.
%	Let Assumptions \ref{assumptions:nonlinear} be satisfied with
%	\begin{equation*}
%	m > \frac{d}{2} + 2.
%	\end{equation*}
%	Then,
%	\begin{enumerate}
%		\item The solution $\mu$ of equation \eqref{target equation} in the sense of Definition \ref{definition of solution} is unique and it is given by $\mu = (\mu_t)_{t\in [0,T]} := (r \mathcal{L}(X_t \mid \mathcal{W}))_{t\in [0,T]}$, where $X_t$ is a solution to equation \eqref{sde system} with initial condition $X_0 \in L^1(\Omega, \mathcal{F}_0; \mathbb{R}^d)$ such that $\mu_0 = r\mathcal{L}(X_0)$.
%		\item Let $\mu^1_0, \mu^2_0 \in \mathcal{M}_r(\mathbb{R}^d)$. %and $X^i_0 \in L^1(\Omega)$ with $\langle \mu^i_0, \varphi\rangle = r\mathbb{E}[\varphi(X^i_0)]$, for $i=1,2$. 
%		The two solutions $\mu^1_t, \mu^2_t$ associated with initial conditions $\mu^1_0, \mu^2_0$ are in $L^1_{\omega}C_t\mathcal{M}_r$ and we have the following stability estimate
%		\begin{equation*}
%				\mathbb{E}\sup_{t\in[0,T]}\rho(\mu^1_t, \mu^2_t) \leq C \rho(\mu^1_0, \mu^2_0),
%		\end{equation*}
%		for some constant $C > 0$.
%	\end{enumerate} 
%\end{theorem}

\begin{proof}
Let $\mu \in L^1_{\omega,t}\mathcal{M}_r$ be a solution to equation \eqref{target equation}, and set $\bar a(t,x) := a(t, x, \mu_t)$, $\bar \sigma(t, x) := \sigma(t, x, \mu_t)$ and $\bar b(t, x) := b(t, x, \mu_t)$. We have that $\bar a, \bar\sigma$ and $\bar b$ are $\mathcal{B}(\mathbb{R}_{+})\times \mathcal{B}(\mathbb{R}^d)\times \mathcal{F}$-measurable and $(\mathcal{F}_t)_{t\in [0,T]}$-adapted processes. It follows from Assumption \ref{assumptions:nonlinear} that Assumption \ref{assumptions:linear} is satisfied by $\bar a, \bar b, \bar \sigma$. % Hence, Theorem \ref{wellposedness linear spde} can be applied and equation \eqref{linear sde system} admits a solution.

Let $X_0 :\Omega \rightarrow \mathbb{R}^d$ be an $\mathcal{F}_0$-measurable random variable such that $\mu_0 = r\mathcal{L}(X_0)$.
Let $X :[0,T]\times \Omega \rightarrow \mathbb{R}^d$ be a time-continuous, $(\mathcal{F}_t)_{t\geq 0}$-adapted solution to equation \eqref{linear sde system} with coefficients given by 
$
(\bar b, \bar \sigma, [2\bar a + (\bar \sigma)^T\bar \sigma]^\frac12)
$ and initial condition $X_0$.
%and let $\mu^2 \in L^1_{\omega}C_t\mathcal{M}_r \subset L^1_{\omega, t}\mathcal{M}_r$ be the solution to the linear equation \eqref{linear ito spde}, with coefficients $b, \sigma, a := a^1$, given by Lemma \ref{existence linear spde}. This solution is characterized as $\langle \mu^2_t, \varphi \rangle = r\mathbb{E}[\varphi(X^2_t) \mid \mathcal{W}]$, for each $\varphi \in C^0(\mathbb{R}^d)$.

Clearly, $\mu$ is also a solution to the linear equation \eqref{linear ito spde} with coefficients $(\bar b, \bar \sigma, \bar a)$. Since $(\bar b, \bar \sigma, \bar a)$ satisfy Assumption \ref{assumptions:linear}, Theorem \ref{wellposedness linear spde} implies the uniqueness of solutions for the linear equation \eqref{linear ito spde} in $L^1_{\omega, t}\mathcal{M}_r$, which implies that $\mu$ corresponds to the solution given by Lemma \ref{existence linear spde}, that is, $\mu = r\mathcal{L}(X\mid \mathcal{W}) \in L^1_{\omega}C_t\mathcal{M}_r$.

This implies that the couple $(X, \mu)$ is a solution to the equation \eqref{sde system} in the sense of Definition \ref{definition solution sde}.

Hence, the solutions of equation \eqref{target equation} are characterized as solutions of the McKean-Vlasov equation \eqref{sde system}, in the sense that, $\mu \in L^1_{\omega,t}\mathcal{M}_r$ is a solution to \eqref{target equation} in the sense of Definition \ref{definition of solution}, if and only if there exists an $(\mathcal{F}_t)_{t\geq0}$-adapted stochastic process $(X_t)_{t\in[0,T]}$, such that $\mu_t = r\mathcal{L}(X_t \mid \mathcal{W})$ and the pair $(X, \mu)$ is a solution to the McKean-Vlasov equation \eqref{sde system} in the sense of Definition \ref{definition solution sde}. 

The uniqueness of solutions to \eqref{target equation} now follows from the uniqueness of solutions to the McKean-Vlasov SDE proven in Theorem \ref{wellposedness sde}.

\end{proof}

\appendix
\section{From Stochastic Scalar Conservation Laws to non-linear stochastic Fokker-Planck equations} \label{appendix:ito strato}
There is a rigorous way to rely the SSCL \eqref{SSCL} to the Fokker-Planck equation \eqref{target equation} using the concept of the Lions derivative in the space $\mathcal{P}_2(\mathbb{R}^d) := \left\{ \mu\in\mathcal{P}(\mathbb{R}^d) \mid \langle\mu,\vert \cdot \vert^2\rangle <+\infty \right\}$ of probability measures with finite second moment, endowed with the $2$-Wasserstein distance. 
The results in this section are taken from \cite{MR3752669, MR3753660}.
Let $(\Omega, \mathcal{F},( \mathcal{F}_t)_{t\geq 0}, \mathbb{P})$ be an atomless filtered probability space, compatible with a $d_1$-dimensional Brownian motion $W$.
We study the following non-local scalar conservation law
\begin{equation}
\label{NLSSCL}
d\mu_t + \operatorname{div}(\mu_t \sigma(x, \mu_t) \circ dW_t) = 0,
\end{equation}
in the following sense
\begin{defn}\label{defn:sol nlsscl}
	We say that a stochastic process $\mu : [0,T] \times \Omega \to \mathcal{P}_2(\mathbb{R}^d)$ is a solution to equation \eqref{NLSSCL} with initial condition $\mu_0 \in \mathcal{P}_2(\mathbb{R}^d)$, if for every $\varphi \in C^2(\mathbb{R}^d)$ the following conditions are satisfied:
	\begin{enumerate}
		\item The process $(\langle \mu_t, \sigma^{i,k}(\mu_t) \partial_i \varphi \rangle)_{t\in[0,T]}$ is an $(\mathcal{W}_t)_{t\geq 0}$-adapted semimartingale;
		\item For Lebesgue-a.e. $t\in[0,T]$ there exists a set of full measure $\Omega^\prime \subset \Omega$ on which the following integral equation is satisfied
		\begin{equation}
		\label{integral nlsscl}
		\langle \mu_t, \varphi \rangle 
		=  \langle \mu_0, \varphi \rangle
		+ \int_{0}^{t} \langle \mu_s, \sigma^{i,k}(\mu_s) \partial_i \varphi \rangle \circ d W^k_s. 
		\end{equation}
	\end{enumerate}
\end{defn}

\begin{asm}
	\label{assumptions:appendix}
	Let $m \in \mathbb{N}$ and assume
	\begin{enumerate}[label=(\roman*), ref=\ref{assumptions:nonlinear} (\roman*)]
		%[label=\textbf{C.\arabic*}]
		\item\label{assumption:sigma Lions appendix} (Lions-differentiability) The function $\sigma(x,\mu) := (\sigma^{i,k}(x,\mu)): \mathbb{R}^d\times \mathcal{P}(\mathbb{R}^d) \to \mathbb{R}^{d \times d_1}$ is measurable and satisfies Assumptions \cite[(Joint Chain Rule Common Noise), p.279]{MR3753660}. In particular, $\sigma$ is twice Lions differentiable in the $\mu$ direction with first derivative $\partial_{\mu}\sigma:  \mathbb{R}^d\times \mathbb{R}^d \times\mathcal{P}(\mathbb{R}^d) \to \mathbb{R}^{d\times d \times d_1}$.
		
		\item\label{assumption:sigma appendix} There exists a positive constant $K_m$ such that for all $\mu \in \mathcal{P}_2(\mathbb{R}^d) $, $\sigma(\cdot,\mu) \in C^{m+1}(\mathbb{R}^d;\mathbb{R}^{d \times d_1})$, $\partial_{\mu}\sigma(\cdot, \cdot,\mu) \in C^{m}(\mathbb{R}^d\times \mathbb{R}^d;\mathbb{R}^{d\times d \times d_1})$  and
		\begin{equation*}
		\sup_{\mu}\| \sigma(\cdot,\mu)\|_{C^{m+1}},
		\;
		\sup_{\mu}\| \partial_{\mu}\sigma(\cdot,\mu)\|_{C^{m}}
		\leq
		K_m.
		\end{equation*}
		
		\item \label{assumptions:appendix Lipschitz} (Uniform Lipschitz continuity) There exists a constant $K > 0$ such that, for all $(x,\mu^\prime), (x,\mu^\prime)\in \mathbb{R}^d \times \mathcal{M}_r(\mathbb{R}^d)$,
		\begin{align*}
		\Vert \sigma(x,\mu) - \sigma(x^\prime,\mu^\prime) \Vert
		\leq K\left(  \vert x - x^\prime \vert + \rho(\mu, \mu^\prime)\right),
		\end{align*}
	\end{enumerate}
\end{asm}

\begin{pro}
	Let $\mu_0 \in \mathcal{P}_2(\mathbb{R}^d)$ and assume that $\sigma$ satisfies Assumption \ref{assumptions:appendix} for some $m > \frac{d}{2} + 2$ and $\sigma$ is independent of the time variable. Then, there exists a $\mathcal{P}_2(\mathbb{R}^d)$-valued solution $\mu$ to equation \eqref{NLSSCL} in the sense of Definition \ref{defn:sol nlsscl}.
	
	Moreover, $\mu$ is the solution to equation \eqref{target equation} with initial condition $\mu_0$ and coefficients $b$ and $\alpha$ given by 
	\begin{equation}
	\label{lions:drift definition}
	b(x, \mu) = (b(x, \mu)^{i}):= \frac12 \sigma^{jk}(t,x, \mu)\partial_j\sigma^{ik}(x, \mu) + \frac 12 G^i(x,\mu),
	\end{equation}
	\begin{equation}
	\label{lions:diffusion definition}
	a (x, \mu) = (a^{ij}(x, \mu)):= \frac12 \sigma^{ik}(x,\mu)\sigma^{jk}(x,\mu),
	\end{equation}
	with $G^i(x,\mu) := \langle \mu, \sigma^{jk}(\cdot, \mu)(\partial_{\mu} \sigma(x, \mu)(\cdot))^{ijk} \rangle$.
\end{pro}

\begin{rem}
	Under Assumption \ref{assumptions:appendix}, with $m\geq 1$, and thanks to \cite[Remark 5.27]{MR3752669} the functions $b,\sigma,a$ satisfy Assumption \ref{assumptions:nonlinear}, with the same $m$.
\end{rem}

\begin{proof}
Given a random variable $X_0 \in L^2(\Omega, \mathcal{F}_0; \mathbb{R}^d)$ with $\mu_0 = \mathcal{L}(X_0)$ consider the following McKean-Vlasov SDE
\begin{equation}
\label{lions:sde}
\left\{
\begin{array}{l}
dX_t = b(X_t,\mu_t)dt + \sigma(X_t,\mu_t) dW_t, \\
X_t|_{t=0} = X_0,\\
\mu_t  := \mathcal{L}(X_t \mid \mathcal{W}).
\end{array}
\right.
\end{equation}

The function $\sigma: \mathbb{R}^d \times \mathcal{P}(\mathbb{R}^d) \to \mathbb{R}^{d\times d_1}$ is given. The coefficient $b: \mathbb{R}^d \times \mathcal{P}(\mathbb{R}^d) \to \mathbb{R}^{d}$ is constructed from $\sigma$ by \eqref{lions:drift definition}.
%The coefficients $\alpha: \mathbb{R}^d \times \mathcal{P}(\mathbb{R}^d) \to \mathbb{R}^{d\times d}$ and $b: \mathbb{R}^d \times \mathcal{P}(\mathbb{R}^d) \to \mathbb{R}^{d}$ are constructed from $\sigma$ by \eqref{lions:drift definition} and \eqref{lions:diffusion definition}. 
Since $(b,\sigma, a)$, with $a$ given by \eqref{lions:diffusion definition}, satisfy Assumption \ref{assumptions:nonlinear}, equation \eqref{lions:sde} is a special case of equation \eqref{sde system}, with $\alpha := (2a - \sigma^T\sigma)^{\frac{1}{2}} = 0$.

If $(X, \mu_t  := \mathcal{L}(X_t \mid \mathcal{W}))$ is the solution to the McKean-Vlasov equation \eqref{lions:sde}, given by Theorem \ref{wellposedness sde} with initial condition $X_0 \in L^2(\Omega, \mathcal{F}_0; \mathbb{R}^d)$, then it is easy to verify that $X_t \in L^2(\Omega, \mathcal{F}_t; \mathbb{R}^d)$. Hence, $\mu_t(\omega) \in \mathcal{P}_2(\mathbb{R}^d)$, for $t\in[0,T]$ and a.e.-$\omega$.

Given a test function $\varphi \in C^2(\mathbb{R}^d)$ and $\nu \in \mathcal{P}_2(\mathbb{R}^d)$, define 
\begin{equation}
\label{defn: u}
u(\nu) = (u^k(\nu))_{k=1,\dots,d_1} := \langle \nu, \sigma(\cdot, \nu)^{i,k}\partial_i \varphi(\cdot) \rangle.
\end{equation}
It follows from Assumption \ref{assumptions:appendix} that $u$ satisfies Assumption \cite[(Simple $C^2$ Regularity), p.268]{MR3753660} and we can apply \cite[equation (4.28)]{MR3753660}. We have
\begin{equation}
\label{lions:ito for u}
u^k(\mu_t) = u^k(\mu_0) 
+ \int_{0}^{t}\langle \mu_s, \sigma^{jl}(\cdot, \mu_s) (\partial_{\mu}u(\mu_s)(\cdot))^{jk} \rangle dW^l_s
+ r^k(t),
\quad \mathbb{P}-a.s.
\end{equation}
where $r(t)$ is a process of bounded variation.
%\begin{align*}
%u(t, x, \mu_t) = & u(0, x, \mu_0) 
%+ \int_{0}^{t}\partial_t u(s, x, \mu_s)ds
%+ \int_{0}^{t}\tilde{\mathbb{E}}^1\left[ \partial_{\mu}u(s, x , \mu_s) (\tilde X_s) \tilde b_s \right] ds\\
%& + \int_{0}^{t}\tilde{\mathbb{E}}^1\left[ \tilde{\sigma}_s^T \partial_{\mu}u(s, x, \mu_s)(\tilde{X}_s) \right] \cdot dW_s
%+ \frac12 \int_{0}^{t} \tilde{\mathbb{E}}^1\left[  \operatorname{trace}\left( \partial_v \partial_{\mu} u(s, x, \mu_s)(\tilde{X}_s) \tilde{\alpha}_s\tilde{\alpha}_s^T\right) \right] ds \\
%& + \frac12 \int_{0}^{t} \tilde{\mathbb{E}}^1\left[  \operatorname{trace}\left( \partial_v \partial_{\mu} u(s, x, \mu_s)(\tilde{X}_s) \tilde{\sigma}_s\tilde{\sigma}_s^T\right) \right] ds \\
%& + \frac12 \int_{0}^{t} \tilde{\mathbb{E}}^1\tilde{\tilde{\mathbb{E}}}^1\left[ \operatorname{trace}\left( \partial_{\mu}^2 u(s, x, \mu_s)(\tilde{X}_s, \tilde{\tilde{X}}_s) \tilde{\sigma}_s\tilde{\tilde{\sigma}}_s^T\right) \right] ds.
%\end{align*}

It is shown in \cite[equation (5.37)]{MR3752669} that
\begin{align}
	(\partial_{\mu}u(\nu)(\cdot))^{jk}
	= & \partial_j\sigma^{ik}(\cdot, \nu)\partial_i\varphi(\cdot) 
	+ \sigma^{ik} (\cdot, \nu)\partial_{i,j}^2\varphi(\cdot)\nonumber\\
	& + \int_{\mathbb{R}^d} (\partial_\mu \sigma(x, \nu)(\cdot))^{ijk} \partial_i \varphi (\cdot) d\nu(x), \quad \forall \nu \in \mathcal{P}_2(\mathbb{R}^d). \label{lions:measure derivative for u}
\end{align} 
%Define the coefficients of \eqref{lions:sde} as
%\begin{equation}
%\label{lions:drift definition}
%b(x, \mu) = (b(x, \mu)^{i}):= \frac12 \sigma^{jk}(t,x, \mu)\partial_j\sigma^{ik}(x, \mu) + \frac 12 G^i(x,\mu),
%\end{equation}
%\begin{equation}
%\label{lions:diffusion definition}
%\alpha (x, \mu) = (\alpha^{ij}(x, \mu)):= \frac12 \sigma^{ik}(x,\mu)\sigma^{jk}(x,\mu),
%\end{equation}
%with $G^i(x,\mu) := \sigma^{jk}(x, \mu)\langle \mu, (\partial_{\mu} \sigma(\cdot, \mu)(x))^{ijk} \rangle$. These coefficients satisfy Assumption \ref{assumptions:nonlinear}.
Due to Theorem \ref{thm:uniqueness spde}, $\mu$ solves equation \eqref{target equation}, in the sense of Definition \ref{definition of solution}. Rewriting the weak formulation yields, $\forall [0, T]$,
\begin{align*}
\langle \mu_t, \varphi \rangle 
= & \langle \mu_0, \varphi \rangle
+ \int_{0}^{t} \langle \mu_s, a^{i,j}(\mu_s) \partial_{i,j}^2 \varphi\rangle ds
+ \int_{0}^{t} \langle \mu_s, b^i(\mu_s) \partial_i \varphi \rangle ds\nonumber\\
& + \int_{0}^{t} \langle \mu_s, \sigma^{i,k}(\mu_s) \partial_i \varphi \rangle d W^k_s,
\quad \mathbb{P}-a.s.
\end{align*}
We rewrite the It\^o integral in terms of a Stratonovich integral. Computing the quadratic covariation between $u(\mu)$ and $W$ ($u$ is defined in \eqref{defn: u}), it follows from \eqref{lions:ito for u} and \eqref{lions:measure derivative for u} that the correction term is
\begin{equation*}
\frac{1}{2}\left[ u(\mu), W \right]_t =
-\frac12 \langle \mu_t, \sigma^{ik}(\mu_t) \sigma^{jk}(\mu_t)\partial_{ij}\varphi  \rangle 
- \frac12 \langle \mu_t, \sigma^{jk}(\mu_t)\partial_j\sigma^{ik}(\mu_t)\partial_i\varphi  \rangle 
- \frac12 \langle \mu_t, G^i(\mu_t) \partial_i \varphi \rangle.
\end{equation*}
By the definition of $b$ and $a$ and cancellations we obtain that $\mu$ satisfies \eqref{integral nlsscl}.
\end{proof}

\section{Remarks on conditional expectation}
Let $(\Omega, \mathcal{A}, \mathbb{P})$ be a fixed probability space.
\begin{lem}
	\label{appendix lemma:conditioning}
	If $\mathcal{F}, \mathcal{G}, \mathcal{H} \subset \mathcal{A}$ are three independent $\sigma$-algebras, and $X$ is a random variable measurable with respect to $\mathcal{F} \vee \mathcal{G}$, then
	\begin{equation*}
	\mathbb{E}\left[ X \mid \mathcal{F} \vee \mathcal{H} \right] = \mathbb{E}\left[ X \mid \mathcal{F}\right], \quad \mathbb{P}-a.s.
	\end{equation*}
\end{lem}
\begin{proof}
	The $\sigma$-algebra $\mathcal{F} \vee \mathcal{G}$ is generated by the sets of the form $F \cap H$, with $F \in \mathcal{F}$ and $H\in\mathcal{H}$. Hence, the following computation concludes the proof
	\begin{equation*}
	\mathbb{E}\left[ \mathbb{E}\left[ X \mid \mathcal{F}\right] \boldsymbol{1}_{F\cap H}\right]
	= \mathbb{E}\left[ \mathbb{E}\left[ X \boldsymbol{1}_{F}\mid \mathcal{F}\right] \boldsymbol{1}_H\right]
	= \mathbb{E}\left[ X \boldsymbol{1}_{F}\right] \mathbb{E}\left[\boldsymbol{1}_H\right]
	= \mathbb{E}\left[ X \boldsymbol{1}_{F}\boldsymbol{1}_H\right].
	\end{equation*}
\end{proof}

\begin{lem}
	\label{lemma:conditional integrals}
	Let $(\Omega, \mathcal{F},(\mathcal{F}_t)_{t\geq0}, \mathbb{P})$ be a filtered probability space compatible with a Brownian motion $W$ and let $B$ be another $(\mathcal{F}_t)_{t\geq 0}$-adapted Brownian motion, which is independent from $W$. Let $(Y_t)_{t\geq 0}$ be an $(\mathcal{F}_t)_{t\geq 0}$-adapted bounded stochastic process. Then, $\forall t \in [0,T]$, $\mathbb{P}-a.s.$,
	\begin{equation*}
	\mathbb{E}\left[\int_{0}^{t} Y_s dB_s \vert \mathcal{W}_t \right] = 0,
	\end{equation*}
	\begin{equation*}
	\mathbb{E}\left[\int_{0}^{t} Y_s dW_s \vert \mathcal{W}_t \right] 
	= \int_{0}^{t}\mathbb{E}\left[ Y_s \vert \mathcal{W}_s\right] dW_s.
	\end{equation*}
\end{lem}
\begin{proof}
	We start with the first equality. Let $\pi^n = \left\{ t_i^n \; : \: i=1,\dots, n  \right\}$ be a sequence of partitions of $[0, t]$ with mesh size going to zero as $n\to 0$, such that
	\begin{equation}\label{integral sequence}
	\lim_{n\to\infty} \mathbb{E} \left \vert \sum_{[t^n_i, t^n_{i+1}] \in \pi^n} Y_{t_i}\left(B_{t_{i+1}} - B_{t_i}\right)
	- \int_{0}^{t} Y_s dB_s \right \vert  \to 0.
	\end{equation}
	As a straightforward consequence of Jensen's inequality, one has that convergence in $L^1$ implies convergence in $L^1$ of the conditional expectations. We can conclude the proof with the following observation
	\begin{equation*}
	 \mathbb{E}\left[ \sum_{[t_i, t_{i+1}] \in \pi^n} Y_{t_i}\left(B_{t_{i+1}} - B_{t_i}\right) \mid \mathcal{W}_t\right] \\
	=  \sum_{[t_i, t_{i+1}] \in \pi^n}  \mathbb{E}\left[B_{t_{i+1}} - B_{t_i}\right]\mathbb{E}\left[ Y_{t_i} \mid \mathcal{W}_t\right] 
	=0,
	\quad \mathbb{P}-a.s.
	\end{equation*}
	Here we used the following property of the conditional expectation: if $B$ is independent of $\sigma(Y, \mathcal{W})$, then $\mathbb{E}[BY \mid \mathcal{W}] = \mathbb{E}[B]\mathbb{E}[Y \mid \mathcal{W}]$.
	
	The second part of the lemma can be proved in a similar way, taking into account the additional observation that, for every $0 \leq s \leq t \leq T$, 
	\begin{equation*}
	\mathbb{E}\left[
	Y_s \mid \mathcal{W}_t
	\right] 
	= 
	\mathbb{E}\left[
	Y_s \mid \mathcal{W}_s
	\right],
	\quad \mathbb{P}-a.s.
	\end{equation*}
	This follows from Lemma \ref{appendix lemma:conditioning}.
\end{proof}
\bibliographystyle{abbrv}
\bibliography{bibliography} 

\end{document}